\documentclass[12pt]{amsart}
\usepackage[english]{babel}
\usepackage[utf8]{inputenc}
\usepackage[foot]{amsaddr}
\usepackage{lipsum}
\usepackage{mathrsfs}
\usepackage{stix}
\usepackage{fullpage}
\usepackage{mathtools}
\usepackage{amsmath}
\usepackage{tikz}
\usepackage{tikz-cd}
\usetikzlibrary{decorations.pathreplacing,angles,quotes}
\usetikzlibrary{arrows,chains,positioning,scopes,quotes}
\newtheorem{defi}{Definition}[section]
\newtheorem{exa}[defi]{Example}

\newtheorem{teo}[defi]{Theorem}
\newtheorem{rem}[defi]{Remark}

\newtheorem{coro}[defi]{Corollary}
\newtheorem{pro}[defi]{Proposition}
\newtheorem*{rem*}{Remark}
\newcommand{\authorfootnotes}{\renewcommand\thefootnote{\@fnsymbol\c@footnote}}%
\makeatother

\newcommand{\interior}[1]{%
  {\kern0pt#1}^{\mathrm{o}}%
}
\newcommand{\MO}{\mathfrak{L}}

\newcommand{\C}{\mathbb{C}}
\newcommand{\Q}{\mathbb{Q}}
\newcommand{\K}{\mathbb{K}}
\newcommand{\R}{\mathbb{R}}
\newcommand{\N}{\mathbb{N}}
\newcommand{\Z}{\mathbb{Z}}

\newcommand{\esp}{\text{  }}

\newcommand\blfootnote[1]{%
  \begingroup
  \renewcommand\thefootnote{}\footnote{#1}%
  \addtocounter{footnote}{-1}%
  \endgroup}

\usepackage[colorlinks=false]{hyperref}
\renewcommand\eqref[1]{(\ref{#1})} 
\begin{document}
\title[HARDY INEQUALITIES ON CONSTANT-ORDER NONCOMMUTATIVE VILENKIN GROUPS] 
 {HARDY INEQUALITIES ON CONSTAN-ORDER NONCOMMUTATIVE VILENKIN GROUPS}

\maketitle
\begin{center}
  
  \normalsize
  \authorfootnotes
  AIDYN KASSYMOV \textsuperscript{1, 2}, J.P.~VELASQUEZ-RODRIGUEZ\textsuperscript{1,*} \par \bigskip

  \textsuperscript{1} \ Ghent University, Ghent, Belgium \par
  \textsuperscript{2} \ Al-Farabi Kazakh National University, Al-Farabi, Kazakhstan\par
\textsuperscript{*} \ Corresponding author\par
  \bigskip

\end{center} 

\setcounter{footnote}{0}
\newcommand{\Addresses}{
 
{
  \bigskip
  \footnotesize
  
  J.P.~VELASQUEZ-RODRIGUEZ \textsc{Department of Mathematics: Analysis, Logic and Discrete Mathematics, Ghent University, Belgium. }\par\nopagebreak
  \textit{E-mail address:} \texttt{juanpablo.velasquezrodriguez@ugent.be}
}

{
  \bigskip
  \footnotesize
  
  AIDYN KASSYMOV, \textsc{Al-Farabi Kazakh National University, 71 Al-Farabi avenue, 050040 Almaty, Kazakhstan  \\  and \\ Department of Mathematics: Analysis, Logic and Discrete Mathematics, Ghent University, Belgium.}\par\nopagebreak
  \textit{E-mail address:} \texttt{aidyn.kassymov@ugent.be}
}

}

\blfootnote{The authors are supported by the FWO Odysseus 1 grant G.0H94.18N: Analysis and Partial Differential Equations. }



\date{\today}
\begin{abstract}
In this note we extend several integral inequalities to the context of noncommutative Vilenkin groups. We prove some sharp weak and strong type estimates for the Hardy operator and the Hardy-Littlewood-P{\'o}lya operator on constant-order noncommutative Vilenkin groups. In particular for graded $\K$-Lie groups, where $\K$ is a non-archimedean local field, we additionally provide some functional inequalities, like the Hardy-Littlewood-Sobolev unequality and the Stein-Weiss inequality, linking some classes of homogeneous pseudo-differential operators, like the Vladimirov-Taibleson operator and the Vladimirov Laplacian, with Hardy inequalities.
\end{abstract}
\maketitle
\section{Introduction }
Between 1915 and 1925 the work of G. H. Hardy and M. Riesz on the Hilbert's inequality and its generalisations produced the well known and celebrated result that we call nowadays \emph{the Hardy inequality} \cite{10.2307/27642033}. Hardy's original theorem states that if $f$ is a measurable function with non-negative values, then it holds $$\int_0^\infty \Big( \frac{1}{x} \int_0^x f(t) dt \Big)^r dx \leq \Big( \frac{r}{r-1} \Big)^r \int_0^\infty f(x)^r dx, \esp \esp 1<r<\infty,$$where the constant $r/(r-1)$ is known to be sharp. The above theorem has inspired a lot of research on similar integral inequalities, and it has been generalised to a wide number of contexts such as homogeneous groups, hyperbolic spaces, complete manifolds, among others. See \cite{10.2307/27642033, Hardyinequalitybook, hardyineqhomogroups, doi:10.1098/rspa.2018.0310} and the references therein for a detailed historical account of the developments on Hardy inequalities. 

In the context of $p$-adic analysis considerable progress on Hardy inequalities has been done during the last 10 years, see \cite{Zun, Hussain, mlinearhardy, hardyoperatorproductspaces, WeightedmultilinearpadicHardyoperatorsandcommutators, p-adiccommutators, p-adicHLSinequality}. In this note we would like to contribute to such progress in two different ways. First, by extending from $\Q_p$ to constant-order noncommutative Vilenkin groups the sharp estimates on the norm of the fractional Hardy operator obtained in \cite{Zun, Hussain}. And second, by extending some of the integrals inequalities and hypoelliptic functional inequalities studied in \cite{2018arXiv180501064R} to graded $\K$-Lie groups, where $\K$ is a non-archimedean local field. An important step in our analysis is the extension of some weighted integral inequalities which will allow us to give more general versions of some of the results given in \cite{Hussain, p-adiccommutators, p-adicHLSinequality}. On top of that, we consider the interplay between Hardy inequalities and certain homogeneous operators, like the Vladimirov-Taibleson operator and the Vladimirov Laplacian, in a similar way as \cite{2018arXiv180501064R} where the reader may find the proof of several hypoelliptic functional inequalities relating Rockland operators with Hardy inequalities. All of our results appear to be new in the literature since, in the knowlege of the authors, there are no works about Hardy inequalities and pseudo-differential operators on noncommutative locally profinite groups, even though locally compact abelian Vilenkin groups like $\Q_p$ have been extensively studied.     

\section{Preliminaries and main results}

Let us start our exposition by recalling the formal definition of the groups we will study, and the particular examples that we will consider.

\begin{defi}\label{defilcvilenkingroup}\normalfont
We say that a topological group $G$ is a \emph{locally compact Vilenkin group} if $G$ is a  locally compact, Hausdorff, totally disconnected, unimodular topological group, endowed with a strictly decreasing sequence of compact open sub-groups $\{G_n\}_{n \in \Z}$ such that \begin{enumerate}
    \item[(i)] It holds:$$\varkappa_n := | G_n /G_{n+1}|< \infty,$$ for every $n \in \Z.$
    \item[(ii)]$$G= \bigcup_{n \in \Z} G_n , \esp \esp \text{and} \esp \esp \bigcap_{n \in \Z} G_n = \{e\}.$$
    \item[(iii)]The sequence $\{G_n\}_{n \in \Z}$ form a basis of neighbourhoods at $e \in G$.
\end{enumerate}
We will say that $G$ is a \emph{bounded-order} Vilenkin group if $$\sup_{n \in \Z} |G_n/G_{n+1}| < \infty.$$ We will say that $G$ is \emph{constant-order} when $|G_n/G_{n+1}|=\varkappa$ for some $\varkappa \in \N$ and every $n \in \Z.$
\end{defi}
\begin{rem}
The group $G$ carries a natural Haar measure. We will assume from now on that this measure is normalized in such a way that $|G_0|= 1$.
\end{rem}

\begin{exa}\label{exaHeisenberg}
Our first example of a locally compact noncommutative Vilenkin group is the $2d+1$-dimensional Heisenberg group over $\Q_p$, here denoted by  $\mathbb{H}_{d}(\Q_p)$ or simply $\mathbb{H}_d$ for short, in turn defined as \[
\mathbb{H}_{d}(\Q_p)= \left\{
  \begin{bmatrix}
    1 & x^t & z \\
    0 & I_{m} & y \\
    0 & 0 & 1 
  \end{bmatrix}\in GL_{d+2}(\Q_p) \esp : \esp x , y \in \Q_p^d, \esp z \in \Q_p \right\}. 
\]
Clearly $\mathbb{H}_{d}(\Q_p)$ is an analytic $2d+1$-dimensional manifold locally homeomorphic to $\Q_p^{2d+1}$. Moreover, the operations on $\mathbb{H}_{d}(\Q_p)$ are analytic functions so $\mathbb{H}_{d}(\Q_p)$ is a $p$-adic Lie group. Let us denote by $\mathfrak{h}_{d}(\Q_p)$ its associated Lie algebra. We can write explicitly \[
\mathfrak{h}_{d}(\Q_p)= \left\{
  \begin{bmatrix}
    0 & X^t & Z \\
    0 & 0_{m} & Y \\
    0 & 0 & 0 
  \end{bmatrix}\in \mathcal{M}_{d+2}(\Q_p) \esp : \esp X , Y \in \Q_p^d, \esp Z \in \Q_p \right\}, 
\]and as usual this Lie algebra is nilpotent. The $p$-adic Heisenberg group is a locally compact Vilenkin group with the filter basis of compact open subgroups $\mathscr{G}:= \{ G_n \}_{n \in \Z}$ given by \[G_n=
p^n \mathbb{H}_{d}( \Z_p):= \left\{
  \begin{bmatrix}
    1 & x^t & z \\
    0 & I_{m} & y \\
    0 & 0 & 1 
  \end{bmatrix}\in GL_{d+2}( \Q_p) \esp : \esp x , y \in p^n \Z_p^d, \esp z \in p^{2n} \Z_p \right\}, 
\]where clearly $|G_n / G_{n+1}| = p^{2d+2}$ for all $n$, so the Heisenberg group is a constant-order Vilenkin group.
\end{exa}

\begin{exa}\label{exatriangularmatrices}
Let $\mathfrak{t}_{m \times m} (\Q_p)$ be the Lie algebra of strictly upper triangular matrices of order $m$ over $\Q_p$, and let $T_{m \times m}(\Q_p)$ be the group of upper uni-triangular matrices over $\Q_p$. $\mathfrak{t}_{m \times m} (\Q_p)$ is a graded algebra with the gradation $$\mathfrak{t}_{m \times m} (\Q_p) =V_1(\Q_p) \oplus ... \oplus V_{m-1
}(\Q_p), \esp \esp \esp \text{where} \esp \esp V_j (\Q_p) = \bigoplus_{j=1}^{m - j} \Q_p E_{i , i+j} . $$Here $E_{i,j}$ denotes the $m \times m $ matrix whose entries are all zero except for the $ij$-entry which is equal to one. It is well known that the exponential $\mathbb{exp}: \mathfrak{t}_{m \times m}(\Q_p) \to T_{m \times m}(\Q_p)$ is bijective and it maps the ideals $$\mathfrak{g}_n:= p^n V_1 (\Z_p) \oplus p^{2n} V_2 ... \oplus p^{(m-1)n}V_{m-1
}(\Z_p), \esp \esp \esp \text{where} \esp \esp V_j (\Z_p) = \bigoplus_{j=1}^{m - j} \Z_p E_{i , i+j},$$to a sequence of compact open subgroups $G_n := \mathbb{exp}(\mathfrak{g}_n)$ that make $T_{m \times m} (\Q_p)$ a locally compact constant-order Vilenkin group. 
\end{exa}

\begin{exa}\label{exaEngelalgebra}
Consider the $4$-dimensional $p$-adic Lie group $\mathbb{E}_4 (\Q_p)$, or $\mathbb{E}_4$ for simplicity, which is defined here as the exponential image of the filiform $\Q_p$-Lie algebra $\mathfrak{e}_4$, called by some authors the Engel algebra, defined in terms of the $\Q_p$-basis $\{X , Y_1 , Y_2 , Y_3\}$ and the commutation relations $$[X, Y_i] = Y_{i+1}, \esp \esp i=1,2, \esp \esp \text{and} \esp \esp \esp \esp [X,Y_3] =0.$$We can think on $\mathfrak{e}_4$ as the matrix algebra containing all the matrices of the form  \[ \begin{bmatrix}
    0 & x & 0 & y_3 \\
    0 & 0 & x &  y_2  \\
    0 & 0 & 0 & y_1 \\ 
    0 & 0 & 0 & 0
  \end{bmatrix}, \esp \esp \esp x , y_1 , y_2 , y_3 \in \Q_p.\]In this way $\mathbb{E}_4$, which is the exponential image of the $\mathfrak{e}_4$, is a nilpotent subgroup of $GL_4 (\Q_p)$ that we call here \emph{the Engel group}. With the coordinates $$(x,y_1,y_2,y_3):= \mathbb{exp}( xX +y_1 Y_1 + y_2 Y_2 + y_3Y_3),$$we can identify $\mathbb{E}_4$ with $\Q_p^4$ endowed with the group law: \begin{align*}
      (x,y_1,y_2 , y_3) \times (x',y_1',y_2' , y_3'):= (x + x',y_1 + y_1',y_2 +y_2' -x y_1',y_3 +  y_3' + \frac{1}{2} x^2 y_1' -x y_2').
  \end{align*}Denote by "$\diamond$" the BCH product on $\mathfrak{e}_4$. Then $\mathbb{E}_4 \cong (\mathfrak{e}_4, \diamond)$ is a constant-order locally compact Vilenkin group since it can be endowed with the sequence of compact open subgroups $(G_n , \times) \cong (p^n \mathfrak{e}_4 , \diamond) $ where $$G_n:= \mathbb{exp}( p^n \mathfrak{e}_4) := \mathbb{exp}( p^n \Z_p X + p^n \Z_p Y_1 + p^{2n} \Z_p Y_2 + p^{3n} \Z_p Y_3) ,$$and it clearly holds $|G_n / G_{n+1}| = p^7$ for all $n \in \Z$. Here $Q=7$ is the homogeneous dimension of $\mathbb{E}_4$.   
\end{exa}
 
\begin{exa}\label{exapadicnilpotentgroups}
More generally, let $p$ be a prime number, $p\neq 2$, and let $\K$ be nonarchimedean local field with ring of integers $\mathcal{O}_\K$, prime ideal $\mathfrak{p} = \mathscr{p} \mathcal{O}_\K$ and residue field $\mathcal{O}_\K / \mathfrak{p} \cong \mathbb{F}_q$, where $q= p^{[\K : \Q_p]}$. Let $\mathfrak{g}$ be a nilpotent $\K$-Lie algebra. We say that $\mathfrak{g}$ is \emph{graded} if there exists a decomposition of the form $$\mathfrak{g} = \bigoplus_{k=1}^{r'} V_{k}, \esp \esp \esp [V_i , V_j]\subseteq V_{i+j}.$$Let $V_{\nu_1},...,V_{\nu_r}$ the non-trivial subspaces appearing in the above decomposition, so that $$\mathfrak{g} = \bigoplus_{k=1}^{r} V_{\nu_k}, \esp \esp b_k := dim(V_{\nu_k}).$$We will call the number $Q:=\nu_1 b_1 + \nu_2 b_2 + ...+ \nu_r b_r$ the homogeneous dimension of $\mathfrak{g}$. In particular, if the elements of the first subspace appearing in the decomposition of $\mathfrak{g}$ generate the whole algebra, we say that $\mathfrak{g}$ is \emph{stratified}. If the Lie algebra of a nilpotent Lie group $G$ is graded, we say that $G$ is a \emph{graded Lie group}. When the Lie algebra is stratified we say that $G$ is a \emph{stratified group}.
A consequence of the graded structure on a $\K$-Lie algebra is the existence of dilations on the group. That is, there exists a family of linear mappings $D_\gamma : \mathfrak{g} \to \mathfrak{g}$, $\gamma \in \K^*$, satisfying :
\begin{itemize}
        \item[-]the mappings $D_\gamma$ are diagonalisable and each $V_j$ is the eigenspace associated to one of the eigenvalues of $D_\gamma$.
        \item[-] each $D_\gamma$ is a morphism of the Lie algebra $\mathfrak{g}$, that is, a linear mapping from $\mathfrak{g}$ to itself which respects the Lie bracket: $$[D_\gamma X , D_\gamma Y] = D_\gamma [X, Y], \esp \esp \text{for all} \esp X, Y \in \mathfrak{g}, \esp \esp \text{and all} \esp \gamma \in \K^*.$$
\end{itemize}
The dilations are transported to the Lie group by the exponential map in the following way: the maps $$\mathbb{exp} \esp  \circ D_\gamma \circ \esp  \mathbb{exp}^{-1}, \esp \esp \gamma \in \K^*  ,$$are automorphisms of the group $G$. We will also denote by $D_\gamma$ the dilations on the group. We may write $$\gamma x := D_\gamma (x).$$
                     
A graded group $G$ can be endowed with the structure of a locally compact Vilenkin group with a sequence compatible with the dilations. For doing so let $X_1 , ...,X_d$ be a $\K$-basis of eigenvectors in $\mathfrak{g}$ associated to the dilations $D_\gamma$ so $D_\gamma$ is diagonal in this basis. Define $$\mathfrak{g}_0 := \mathcal{O}_\K X_1 + ... + \mathcal{O}_\K X_d, \esp \esp \text{and} \esp \esp G_0 := \mathbb{exp}(\mathcal{O}_\K X_1 + ... + \mathcal{O}_\K X_d). $$Hence, if $|\gamma|_\K = q^{-n}$, $$G_n := \mathbb{exp}( \gamma \mathfrak{g}_0) = D_\gamma ( G_0)  , \esp \esp  \gamma \mathfrak{g}_0 =  \mathscr{p}^{n \nu_1 } \mathcal{O}_\K X_1 + ... + \mathscr{p}^{n \nu_r} \mathcal{O}_\K X_d = D_{\gamma} (\mathfrak{g}_0),$$defines a sequence $\{ G_n \}_{n \in \Z}$ of compact open sub-groups of $G$ such that $$G= \bigcup_{n \in \Z} G_n , \esp \esp \esp \text{and} \esp \esp \esp \bigcap_{n \in \Z} G_n = \{ e \}.$$
With the above filtration $G$ is a constant Vilenkin group since \begin{align*}
    | G_n / G_{n+1}| & = | \mathfrak{g}_n / \mathfrak{g}_{n+1}| \\ & = q^{\nu_1 b_1 + ... + \nu_r b_r} \\ &= q^{Q},
\end{align*} for all $n \in \Z$. 
\end{exa}

\begin{rem}The $p$-adic Lie groups $\Q_p^d ,\mathbb{E}_4 (\Q_p), \mathbb{H}_{d}(\Q_p)$ and $T_{m \times m} (\Q_p)$ are examples of graded $\K$-Lie groups. As for the real case, the standard dilations are given by:\begin{itemize}
    \item $D_\gamma (x_1,...,x_d) = (\gamma x_1 , ... , \gamma x_d)$, for $(x_1,...,x_d) \in \Q_p^d$.
    \item $D_\gamma (x, y_1 , y_2,y_3) = ( \gamma x , \gamma y_1 , \gamma^2 y_2 , \gamma^3 y_3)$, for $(x, y_1 , y_2,y_3) \in \mathbb{E}_4 (\Q_p)$.
    \item $D_\gamma (x, y , z) = (\gamma x, \gamma y, \gamma^2 z)$, for $(x,y,z) \in \mathbb{H}_{d} (\Q_p).$
    \item $ [ D_\gamma (M) ]_{ij} = \gamma^{j-i} [M]_{ij}, \esp \esp 1 \leq i < j \leq m $, for $M \in T_{m \times m} (\Q_p).$
\end{itemize} 
\end{rem}

We will be using the following notation along the exposition of our results. 
\begin{defi}\normalfont\label{basicdefinitions}
\esp 
\begin{enumerate}
    \item[(i)] Let $w$ be a non-negative locally integrable function and let $1 \leq r <  \infty$ be a real number. We will use the notation $L^r_w (G)$ for the weighted $L^r$-spaces $L^r(G, w)$. In particular $L^r_\alpha (G)$ will denote the weighted $L^r$-spaces $L^r (G , | \cdot | ^\alpha)$, $\alpha \in \R$. Also, we will denote by $\Tilde{L}^r_\alpha (G)$ the collection of radial functions in $L^r_\alpha (G)$.
    \item[(ii)] Let $A$ be a measurable set. We will denote by $\mathbb{1}_A$ the characteristic function of the set $A$.
    \item[(iii)] Given two Banach spaces $E$ and $F$ we will denote by $\mathcal{L} (E,F)$ the collection of bounded operators from $E$ to $F$. In particular when $E=F$ we will write $\mathcal{L}(E)$ instead of $\mathcal{L}(E,E)$. We will write $\| T \|_{\mathcal{L} (E,F)}$ for the operator norm of a linear operator $T \in \mathcal{L} (E, F)$. 
\end{enumerate}
\end{defi}

\subsection{The norm of the Hardy operator}
A simple but important example of a Vilenkin group is the additive group of the $d$-dimensional vector space over the $p$-adic numbers $\Q_p^d$, where $p$ is prime number that we fix for the rest of the work. The field of $p$-adic numbers $\Q_p$ is defined as the complection of the field of rational numbers $\Q$ with respect to the $p$-adic norm $|\cdot|_p$ which is defined as \[|x|_p := \begin{cases}
0 & \esp \text{if} \esp x=0, \\ p^{-\gamma} & \esp \text{if} \esp x= p^{\gamma} \frac{a}{b},
\end{cases}\]where $a$ and $b$ are integers coprime with $p$. The integer $\gamma= ord(x)$, with $ord(0) := + \infty$, is called the $p$-adic order of $x$. The field $\Q_p$ with the norm $|\cdot|_{p}$ is an ultrametric space, that is, the norm $|\cdot|_p$ satisfy the ultrametric inequality: $$|x-y|_p \leq \max\{ |x|_p , |y|_p \}, \esp \esp x, y \in \Q_p.$$A consequence of this property is that the balls $B(0,p^k)$, $k \in \Z$, are compact open subgroups of $\Q_p$. The unit ball of $\Q_p$ with the $p$-adic norm is called the compact group of $p$-adic integers and it will be denoted by $\Z_p$. Any $p$-adic number $x \neq 0$ has a unique expansion of the form $$x = p^{ord(x)} \sum_{j=0}^{\infty} x_j p^j,$$where $x_j \in \{0,1,...,p-1\}$ and $x_0 \neq 0$. By using this expansion we can see that the metric balls determined by $| \cdot |_p$ can be writen as $B(0, p^{-n}) = p^n \Z_p$. The ultrametric on $\Q_p$ can be extended to $\Q_p^d$ by defining the norm $$|x|_p := \max_{1 \leq i \leq d} |x_i|_p.$$The ultrametric induced by this norm determines a natural sequence of compact open sub-groups of $\Q_p^d$, the metric balls$$G_n := B (0,p^{-n}) = \{x \in \Q_p^d \esp : \esp |x|_p \leq p^{-n}\}.$$ The authors in \cite{Zun, Hussain} have used the metric induced by $| \cdot |_p$ to define and study the $p$-adic fractional Hardy operator on $G=\Q_p^d$ introduced in \cite{p-adiccommutators} and defined as $$H_{\Q_p^d}^\delta f(x) := \frac{1}{|x|_p^{d - \delta}} \int_{|y|_p \leq |x|_p} f(y) dy , \esp \esp x \in \Q_p^d \setminus \{ 0 \}. $$In that setting they have provided sharp weak and strong estimates for the fractional Hardy operator acting on weighted $L^r$-spaces. These estimates do not depend on any particular property of the group $\Q_p^d$, and the arguments in \cite{Zun, Hussain} can be extended, in principle, to constant-order Vilenkin groups. The most remarkable change that we have to introduce in the arguments given in \cite{Zun, Hussain} in order to achieve such a generalisation is the choice of metric, and consequently the choice of definition of the Hardy operator. Here in this work we will use the natural metric on Vilenkin groups given in terms of a sequence of compact open sub-groups.
\begin{defi}\normalfont\label{metricvilenkingroups}
Let $G$ be a locally compact Vilenkin group with a sequence of compact open sub-groups $\mathscr{G}= \{ G_n\}_{n \in \N_0}$. For $x,y \in G$ define their associated distance by
\[\varrho_\mathscr{G} (x,y) = |x y^{-1}|_{\mathscr{G}} :=\begin{cases} 0 & \esp \esp \text{if} \esp x=y, \\ |G_n|  & \esp \esp \text{if} \esp x y^{-1} \in G_n \setminus G_{n+1}.\end{cases}\]  
\end{defi}

In some Vilenkin groups it might be more natural for one to choose a certain particular sequence of compact open subgroups and a particular distance function. This is the case of graded $\K$-Lie groups where we take the sequence to be the one determined by the dilations of the group so that we can define a convenient ultrametric that turns out to be an homogeneous quasi-norm.   

\begin{defi}\label{defihomoquasinorm}\normalfont
Let $G$ be a graded $\K$-Lie group together with the dilations $D_\gamma, \gamma \in \K^*$. An \emph{homogeneous quasi-norm} is a continous non-negative function $| \cdot| : G \to [0 , \infty)$ satisfying 
\begin{itemize}
    \item $|x^{-1}| = |x|$ for all $x \in G$,
    \item $|\gamma x| = | \gamma |_\K |x| $ for all $x \in G$ and $\gamma \in \K^*$,
    \item $|x|=0$ if and only if $x = e$.
\end{itemize}
\end{defi}
\begin{defi}\label{defimetricgradedgroups}\normalfont
Let $G$ be a graded $\K$-Lie group with homogeneous dimension $Q$, considered as a Vilenkin group with the sequence of subgroups $\mathscr{G}=\{ G_n\}_{n \in \Z}$ determined by the dilations on the group. We define the distance function $| \cdot |_G$ on $G$ by  \[ = |x y^{-1}|_{G} :=\begin{cases} 0 & \esp \esp \text{if} \esp x=y, \\ |x y^{-1}|_\mathscr{G}^{1/Q}  & \esp \esp \text{if} \esp x y^{-1} \in G_n \setminus G_{n+1}.\end{cases}\]
\end{defi}
\begin{rem}
The above defined metric is an homogeneous quasi-norm in the sense of Definition \ref{defihomoquasinorm}. To see this take $xy^{-1} \in G_n \setminus G_{n+1}$ and $\gamma \in \K^*$ with $|\gamma|_\K = q^{-k}$. Thus $\gamma(xy^{-1}) \in G_{n+k} \setminus G_{n+k+1}$ so $|\gamma(xy^{-1})|_{\mathscr{G}}=q^{-Q(n+k)} = |\gamma|_\K^{Q} |xy^{-1}|_\mathscr{G}$.
\end{rem}

\begin{rem}
Notice that the metric on the group depends on the choice of the sequence of sub-groups. However we will simply write $| \cdot |$ for the metric on the group since the choice of the sequence of sub-groups will be clear from the context. Also, for simplicity, in this work we will only consider constant-order Vilenkin groups. In that case the possible values taken by the distance function will always be powers of a certain particular number $\varkappa$. In principle this $\varkappa$ is a natural number but actually we might take it as any real number greater than one by defining the metric on the group as \[\varrho_\mathscr{G}' (x,y) = |x y^{-1}|_{\mathscr{G}}' :=\begin{cases} 0 & \esp \esp \text{if} \esp x=y, \\ \varkappa^{-n}  & \esp \esp \text{if} \esp x y^{-1} \in G_n \setminus G_{n+1}.\end{cases}\]  

\end{rem}
With the above defined metric we can define the fractional Hardy operator on Vilenkin groups in a natural way. 
\begin{defi}\normalfont\label{defihardyoperator}
\esp 
\begin{itemize}
    \item Let $G$ be a locally compact Vilenkin group. Define the \emph{fractional Hardy operator} on $G$ with respect to the ultrametric $| \cdot |$ by $$H_\delta f (x) := \frac{1}{|x|^{1 - \delta}} \int_{B(e,|x|)} f(t) dt, \esp \esp x \in G \setminus \{ e \}. $$Here $0 \leq \delta < 1$ is a real number. The \emph{adjoint fractional Hardy operator} is defined as $$H_\delta^* f(x) := \int_{|y|>|x|} \frac{f(y)}{|y|^{1 - \delta}} dy, \esp \esp x \in G \setminus \{ e \}.$$ In particular when $\delta=0$ we will simply write $H$ and $H^*$ instead of $H_0$ and $H_0^*$. 
    \item Let $G$ be a graded $\K$-Lie group. Define the \emph{fractional Hardy operator} on $G$ with respect to the ultrametric $| \cdot |_G$ by $$H_\delta f (x) := \frac{1}{|x|_G^{Q - \delta}} \int_{B(e,|x|_G)} f(t) dt, \esp \esp x \in G \setminus \{ e \}, $$where $ 0 \leq \delta < Q$ is a real number. The \emph{adjoint fractional Hardy operator} is defined as $$H_\delta^* f(x) := \int_{|y|_G>|x|_G} \frac{f(y)}{|y|_G^{Q - \delta}} dy, \esp \esp x \in G \setminus \{ e \}.$$
\end{itemize}
\end{defi}

By using similar techniques as the authors in \cite{Zun, Hussain} we compute explicitly the norm of the previously defined fractional Hardy operator acting on radial functions, for a constant-order Vilenkin group $G$. 
\begin{teo}\label{teosharpestimateradialfunct}
Let $G$ be a constant-order Vilenkin group, let us say $|G_n / G_{n+1}| = \varkappa$ for all $n \in \Z$. Let $1<r<\infty$, $0 \leq  \delta < 1$ and $- \infty < \alpha < r(1 - \delta)-1$ be real numbers. Then $$\|H_\delta \|_{\mathcal{L}(\Tilde{L}^r_{\alpha+ \delta r} (G) , \Tilde{L}^r_\alpha (G))} = \frac{1 - \varkappa^{-1}}{1 - \varkappa^{\frac{\alpha}{r} - \frac{1}{r'} + \delta}},$$where $\frac{1}{r} + \frac{1}{r'} =1$. In particular when $\delta=0$ $$\|H \|_{\mathcal{L}( \Tilde{L}^r_\alpha (G))} = \frac{1 - \varkappa^{-1}}{1 - \varkappa^{\frac{\alpha}{r} - \frac{1}{r'} }}.$$ 
\end{teo}

In particular for graded $\K$-Lie groups we can provide a much better result:

\begin{coro}\label{corostronghardyhomopadicLie}
Let $G$ be a graded $\K$-Lie group with homogeneous dimension {\normalfont{$\emph{Q}$}}. Let $1<r<\infty$ and $- \infty < \alpha < (r-1)Q - \delta r $. Then the norm of the fractional Hardy operator in $\mathcal{L} (L^r_{\alpha + \delta r } (G) , L^r_\alpha (G))$ equals the norm of its restriction to radial functions and $$\|H_\delta \|_{\mathcal{L}(L^r_{\alpha + \delta r } (G) , L^r_\alpha (G))} = \frac{1 - q^{-\normalfont{\emph{Q}}}}{1 - q^{\frac{\alpha}{r} - \frac{Q}{r'} + \delta}},$$where $\frac{1}{r} + \frac{1}{r'} =1$.
\end{coro}

\begin{rem}
It is important to point out that in all the statements about constant-Vilenkin groups the weighted space $L^r_\alpha (G)$ is defined in terms of the ultrametric function $|\cdot |_\mathscr{G}$ described in Definition \ref{metricvilenkingroups}, while for graded $\K$-Lie groups we always consider the weighted space $L^r_\alpha (G)$ as defined in terms of the homogeneous quasi-norm introduced in Definition \ref{defimetricgradedgroups}. 
\end{rem}

In the same way, we have the following estimates for the adjoint operator: 
\begin{teo}\label{teosharpstrongfractionaladjointHardyradial}
Let $G$ be a constant-order Vilenkin group, let us say $|G_n / G_{n+1}| = \varkappa$ for all $n \in \Z$. Let $1<r<\infty$, $0 \leq \delta < 1$ and $-1 < \alpha < \infty $ be real numbers. Then $$\|H_\delta^* \|_{\mathcal{L}(\Tilde{L}^r_{\alpha+ \delta r} (G) , \Tilde{L}^r_\alpha (G))} = \frac{1 - \varkappa^{-1}}{ \varkappa^{\frac{\alpha+1}{r}} -1}.$$ In particular when $\delta=0$ $$\|H^* \|_{\mathcal{L}( \Tilde{L}^r_\alpha (G))} = \frac{1 - \varkappa^{-1}}{ \varkappa^{\frac{\alpha + 1}{r}} - 1}.$$
\end{teo}
Similar to Corollary \ref{corostronghardyhomopadicLie}, with a suitable change of variable we can prove an analogous estimate for the adjoint operator:
\begin{coro}\label{corostrongadjointhardypadicLiehomogeneous}
Let $G$ be a graded $\K$-Lie group with homogeneous dimension {\normalfont{$\emph{Q}$}}. Let $1<r<\infty$ and $ - Q < \alpha < \infty $ be real numbers. Then the norm of the adjoint fractional Hardy operator in $\mathcal{L} (L^r_{\alpha + \delta r } (G) , L^r_\alpha (G))$ equals the norm of its restriction to radial functions and $$\|H_\delta^* \|_{\mathcal{L}(L^r_{\alpha + \delta r } (G) , L^r_\alpha (G))} = \frac{1 - q^{-\normalfont{\emph{Q}}}}{q^{\frac{\alpha +Q}{r}} - 1}.$$
\end{coro}

For the fractional Hardy operator it is also possible to estimate its operator norm when acting on weak $L^r$-spaces:

\begin{defi}\normalfont\label{defiweakspaces}
Let $w$ be a non-negative locally integrable function. The weighted weak Lebesgue space $L^{r, \infty}_w (G)$ is defined as  the  set  of  all  measurable  functions $f$ belonging to $L^1_{loc} (G)$ and satisfying the condition

$$\| f \|_{L^{r, \infty}_w (G)} := \sup_{\lambda >0} \lambda \cdot w \Big(\{ x \in G \esp : \esp |f(x)|> \lambda\} \Big)^{1/r} < \infty.$$Here, for a measurable set $A$, it is defined $$w(A) := \int_A w(x) dx. $$
\end{defi}

For the norm of the fractional Hardy operator we have the following estimate:

\begin{teo}\label{teoweakestHardy}
Let $G$ be a constant-order Vilenkin group, let us say $|G_n/G_{n+1}|=\varkappa$ for all $n \in \Z$. Let $1<r< \infty,$ $1 \leq s < \infty,$ $-\infty < \beta<r-1$ and  $0 \leq \delta < (\beta + 1)/r$ be real numbers, and take $\gamma \in \R$ such that $\frac{\beta + 1}{r} - \delta = \frac{\gamma+1}{s}$. Then $$\| H_\delta \|_{\mathcal{L}(L^r_\beta (G), L^{s,\infty}_\gamma (G))} = \Big( \frac{1 - \varkappa^{-1}}{1- \varkappa^{-(\gamma + 1)}} \Big)^{1/s} \Big( \frac{1 - \varkappa^{-1}}{1- \varkappa^{\frac{\beta}{r-1} -1}}  \Big)^{1/r'} = \Big( \frac{1 - \varkappa^{-1}}{1- \varkappa^{-s(\frac{\beta +1}{r} - \delta )}} \Big)^{1/s} \Big( \frac{1 - \varkappa^{-1}}{1- \varkappa^{\frac{\beta}{r-1} -1}}  \Big)^{1/r'},$$where $\frac{1}{r} + \frac{1}{r'} =1.$
\end{teo}

Notice that in the case where $r=s$ and $\delta = 0$ the value of $\gamma$ has to be equal to $\beta$. It follows: 
\begin{coro}
Let $G$ be a constant-order Vilenkin group, let us say $|G_n / G_{n+1}| = \varkappa$ for all $n \in \Z$. Let $1<r< \infty$ and $- \infty <  \beta < r-1$ be given real numbers. Then: $$\| H \|_{\mathcal{L}(L^r_\beta (G), L^{r,\infty}_\beta (G))} = \Big( \frac{1 - \varkappa^{-1}}{1- \varkappa^{-(\beta + 1)}} \Big)^{1/r} \Big( \frac{1 - \varkappa^{-1}}{1- \varkappa^{\frac{\beta}{r-1} -1}}  \Big)^{1/r'},$$where $\frac{1}{r} + \frac{1}{r'} =1.$ In particular when $\beta=0$: $$\| H \|_{\mathcal{L}(L^r (G), L^{r,\infty} (G))} =1.$$
\end{coro}

In the special case  $r=1$ we got a similar estimation: 

\begin{teo}\label{teoweakHardyL1}
Let $G$ be a constant-order Vilenkin group, let us say $|G_n/G_{n+1}|=\varkappa$ for all $n \in \Z$. Let $1 \leq s< \infty,$ $0 \leq  \delta < 1$ and $ - \infty < \beta <  1 - \delta$ be real numbers, and take $\gamma \in \R$ such that $1 - \delta - \beta = \frac{\gamma+1}{s}$. Then $$\| H_\delta \|_{\mathcal{L}(L^1_{-\beta} (G), L^{s,\infty}_\gamma (G))} \leq \Big( \frac{1 - \varkappa^{-1}}{1- \varkappa^{-(\gamma + 1)}} \Big)^{1/s}= \Big( \frac{1 - \varkappa^{-1}}{1- \varkappa^{-s(1 - \delta - \beta)}} \Big)^{1/s} .$$In particular when $\beta=0$ it holds $$\| H \|_{\mathcal{L}(L^1 (G), L^{s,\infty}_\gamma (G))}= \frac{1 - \varkappa^{-1}}{1 - \varkappa^{- (\gamma+ 1)}}=\Big( \frac{1 - \varkappa^{-1}}{1- \varkappa^{-s(1 - \delta)}} \Big)^{1/s} .$$
\end{teo}
\begin{coro}
Let $s=1$ and $\delta =0$. Then $\gamma = - \beta$ and in consequence $$\| H \|_{\mathcal{L}( L^1_{-\beta} (G), L^{1 , \infty}_{-\beta} (G) )} \leq \Big( \frac{1 - \varkappa^{-1}}{1 - \varkappa^{-(1 - \beta)}} \Big)^{1/s}$$ 
\end{coro}

For graded $\K$-Lie groups the previous results take the following special form: 

\begin{coro}\label{coroweakestHardyhomogeneousliegroups}
Let $G$ be a graded $\K$-Lie group with homogeneous dimension {\normalfont{$\emph{Q}$}}. Let $1<r< \infty,$ $1 \leq s < \infty,$ $-\infty < \beta<(r-1)Q$ and  $0 \leq \delta < (\beta + Q)/r$ be real numbers, and take $\gamma \in \R$ such that $\frac{\beta + Q}{r} - \delta = \frac{\gamma+Q}{s}$. Then $$\| H_\delta \|_{\mathcal{L}(L^r_\beta (G), L^{s,\infty}_\gamma (G))} = \Big( \frac{1 -q^{- \normalfont{\emph{Q}}}}{1- q^{-(\gamma + Q)}} \Big)^{1/s} \Big( \frac{1 - q^{-Q}}{1- q^{\frac{\beta}{r-1} -Q}}  \Big)^{1/r'} = \Big( \frac{1 - q^{-Q}}{1- q^{-s(\frac{\beta +Q}{r} - \delta )}} \Big)^{1/s} \Big( \frac{1 - q^{-Q}}{1- q^{\frac{\beta}{r-1} -Q}}  \Big)^{1/r'},$$where $\frac{1}{r} + \frac{1}{r'} =1.$
\end{coro}

\begin{coro}\label{coroweakHardyhomogeneousL1}
Let $G$ be a graded $\K$-Lie group with homogeneous dimension {\normalfont{$\emph{Q}$}}. Let $1 \leq s< \infty,$ $0 \leq  \delta < Q$ and $ - \infty < \beta <  Q - \delta$ be real numbers, and take $\gamma \in \R$ such that $Q - \delta - \beta = \frac{\gamma+Q}{s}$. Then $$\| H_\delta \|_{\mathcal{L}(L^1_{-\beta} (G), L^{s,\infty}_\gamma (G))} \leq \Big( \frac{1 - q^{-Q}}{1- q^{-(\gamma + Q)}} \Big)^{1/s}= \Big( \frac{1 - q^{-Q}}{1- q^{-s(Q - \delta - \beta)}} \Big)^{1/s} .$$In particular when $\beta=0$ it holds $$\| H \|_{\mathcal{L}(L^1 (G), L^{s,\infty}_\gamma (G))}= \frac{1 - q^{-Q}}{1 - q^{-(\gamma+ Q)}}=\Big( \frac{1 - q^{-Q}}{1- q^{-s(Q - \delta)}} \Big)^{1/s} .$$
\end{coro}
A similar results holds for the adjoint operator: 

\begin{coro}\label{coroweakestadjointHardyhomogeneous}
Let $G$ be a graded $\K$-Lie group with homogeneous dimension {\normalfont{$\emph{Q}$}}. Let $1<r< \infty,$ $1 \leq s < \infty,$ $ 0 \leq \delta < Q$ and $\delta r - Q < \beta < \infty$ be real numbers, and take $\gamma \in \R$ such that $\frac{\beta + Q}{r} - \delta = \frac{\gamma+Q}{s}$. Then $$\| H_\delta^* \|_{\mathcal{L}(L^r_\beta (G), L^{s,\infty}_\gamma (G))} = \Big( \frac{1 - q^{-\normalfont{\emph{Q}}}}{1- q^{-(\gamma + Q)}} \Big)^{1/s} \Big( \frac{1 - q^{-Q}}{q^{(\frac{\beta + Q}{r } - \delta ) r' } -1}  \Big)^{1/r'},$$where $\frac{1}{r} + \frac{1}{r'} =1.$
\end{coro}

Finally, we include in our calculations a sharp estimate for the so called Hardy-Littlewood-P{\'o}lya operator introduced in \cite{article}.

\begin{defi}\label{defiHLPoperator}\normalfont
The Hardy-Littlewood-P{\'o}lya operator on $G$ is defined as $$T f(x) = \int_{G \setminus \{e\}} \frac{f(y)}{\max \{ |x| , |y| \}} dy, \esp \esp x \in G \setminus \{ e \}.$$
\end{defi}
For the Hardy-Littlewood-P{\'o}lya operator we calculate its operator norm when acting on radial functions: 
\begin{teo}\label{TeoHLPradialfunctions}
Let $G$ be a constant-order Vilenkin group, let us say $|G_n / G_{n+1}| = \varkappa$ for all $n \in \Z$. Let $1<r<\infty$ and $- \infty < \alpha < r-1$. Then $$\| T\|_{\mathcal{L}(\Tilde{L}^r_\alpha (G))} = (1-\varkappa^{-1})\Big(\frac{1}{1-\varkappa^{\frac{\alpha}{r} - \frac{1}{r'}}} + \frac{\varkappa^{-\frac{\alpha + 1}{r}}}{1-\varkappa^{-\frac{\alpha+1}{r}}} \Big),$$where $\frac{1}{r} + \frac{1}{r'} =1$.
\end{teo}

\begin{coro}\label{coroHLPp-adicLiehomogeneousn}
Let $G$ be a graded $\K$-Lie group of dimension $d:=dim(G)$ and homogeneous dimension {\normalfont{$\emph{Q}$}}. Let $1<r<\infty$ and $\alpha < r-1$. Then $$\| T\|_{\mathcal{L}(L^r_\alpha (G))} = (1-q^{-\normalfont{\emph{Q}}})\Big(\frac{1}{1-q^{\frac{\alpha}{r} - \frac{Q}{r'})}} + \frac{q^{-\frac{\alpha + Q}{r}}}{1-q^{-\frac{\alpha+Q}{r}}} \Big),$$where $\frac{1}{r} + \frac{1}{r'} =1$. 
\end{coro}

\subsection{Functional inequalities on graded $\K$-Lie groups groups}
Some $\K$-Lie groups can be endowed with an extra structure that reassembles the homogeneous structures on Lie groups. These are the graded $\K$-Lie groups in Example \ref{exapadicnilpotentgroups} where a gradation on their Lie algebra allow us to talk about dilations on the group. The existence of dilations implies some kind of polar decomposition so it is reasonable to think that the arguments given in \cite{doi:10.1098/rspa.2018.0310} for metric measure spaces would also work for graded $\K$-Lie groups. We will show that indeed we can prove similar integral inequalities in this context but we will need to modify slightly the arguments, so that our results are not exactly a particular case of the results in \cite{doi:10.1098/rspa.2018.0310} for metric measure spaces, even though we will be working within an \emph{ultrametric measure space}. Also the constants appearing in our results are quite different since they often come from the calculation of a certain geometric series. Anyway, we can still consider our results as an extension of the well known Hardy inequalities to graded $\K$-Lie groups, where all these results appear to be new and, in addition, we can study the interaction between Hardy inequalities and certain pseudo-differential operators like the Vladimirov-Taibleson operator and the Vladimirov Laplacian on graded $\K$-Lie groups, which is a research topic that also appears to be absent in the literature. The Vladimirov-Taibleson operator provides a suitable analogue of the fractional Laplacian on totally disconnected groups so, keeping in mind the rich literature on Hardy inequalities with differential operators on real Lie groups, it is natural to wonder about the connections between Hardy inequalities and pseudo-differential operators on totally disconnected groups. It is our purpose to give some light about this topic in this section, where we prove the Hardy-Littlewood-Sobolev unequality, the Stein-Weiss inequality, the Gagliardo-Nirenberg inequality and a version of the uncertainty principle on graded $\K$-Lie groups. The techniques we will use to do that are pretty much the same as the ones used to study Hardy inequalities with Rockland operators so, somehow the results in this section are analogues of some already known theorems on graded real Lie groups.

We begin with our work in this direction by proving in Theorem \ref{teofirstintegralinequalitycompact} and Theorem \ref{Teofirstintegralineqnoncompact} a weighted $L^r-L^s$ integral inequality for the Hardy operator on graded $\K$-Lie groups. Using this theorems we will be able to obtain later some interesting results like the Stein-Weiss inequality and the Hardy-Littlewood-Sobolev inequality corresponding to the present framework, and some other functional inequalities. An important remark worth to make now is that we will only deal here with weight functions that are homogeneous. This will be enough for our purposes but of course the question of how to extend  the results to non-homogeneous weights still needs to be adressed.

\begin{teo}\label{teofirstintegralinequalitycompact}
Let $G$ be a compact graded $\K$-Lie group with homogeneous dimension $Q$ and let $\varphi , \psi :G \to [0 , \infty)$ be continuous homogeneous functions of homogeneous degree $- \alpha$ and $\beta$, $ \alpha >Q,  \beta < Q/(r'-1),$ respectively. Let $r,s$ be positive real numbers such that $1< r \leq s < \infty$. Then the condition \begin{equation}\label{eq2}
\frac{\alpha -Q}{s} = \frac{Q}{r'} - \frac{\beta}{r},
\end{equation}is sufficient for the following inequality to hold for any function $f \geq 0$  a.e. on $G$:\begin{equation}\label{eq1}
    \Big( \int_G  \Big( \int_{B(e, |x|_G)} f(z) dz \Big)^s\varphi (x)dx \Big)^{1/s} \leq C \Big(\int_G f(x)^r \psi (x) dx \Big)^{1/r}.
\end{equation}Moreover, for the above inequality to hold the following condition is necessary: \begin{equation}\label{equat3}
    \frac{\alpha -Q}{s} \leq \frac{Q}{r'} - \frac{\beta}{r}.
\end{equation}
\end{teo}
\begin{teo}\label{Teofirstintegralineqnoncompact}
Let $G$ be a graded $\K$-Lie group with homogeneous dimension $Q$ and let $\varphi , \psi :G \to [0 , \infty)$ be continuous homogeneous functions of homogeneous degree $- \alpha$ and $\beta$, $ \alpha >Q, 0 \leq \beta <Q/(r'-1),$ respectively. Let $r,s$ be positive real numbers such that $1< r \leq s < \infty$. Then the inequality \begin{equation}\label{eq3}
    \Big( \int_G  \Big( \int_{B(e, |x|_G)} f(z) dz \Big)^s\varphi (x)dx \Big)^{1/s} \leq C \Big(\int_G f(x)^r \psi (x) dx \Big)^{1/r},
\end{equation}holds for any function $f \geq 0$  a.e. on $G$ if and only if we have \begin{equation}\label{eq4}
\frac{\alpha -Q}{s} = \frac{Q}{r'} - \frac{\beta}{r}.
\end{equation}
\end{teo}

And a similar inequality holds for the adjoint operator.

\begin{teo}\label{Teosecondintegralineqnoncompact}
Let $G$ be a graded graded $\K$-Lie group with homogeneous dimension $Q$ and let $\varphi , \psi :G \to [0 , \infty)$ be continuous homogeneous functions of homogeneous degree $- \alpha$ and $\beta$, $ \alpha <Q, \beta > Q/(r' -1) ,$ respectively. Let $r,s$ be positive real numbers such that $1< r \leq s < \infty$. Then the inequality \begin{equation}\label{eq5}
    \Big( \int_G  \Big( \int_{G \setminus B(e, |x|_G)} f(z) dz \Big)^s\varphi (x)dx \Big)^{1/s} \leq C \Big(\int_G f(x)^r \psi (x) dx \Big)^{1/r},
\end{equation}holds for any function $f \geq 0$  a.e. on $G$ if and only if we have \begin{equation}\label{eq6}
\frac{\alpha -Q}{s} = \frac{Q}{r'} - \frac{\beta}{r}.
\end{equation}
\end{teo}

As an application of the previous weighted integral inequalities we obtain the following integral version of the Hardy inequality:
\begin{teo}\label{Teofunctionalineq}
Let $G$ be a graded $\K$-Lie group with homogeneous dimension $Q$. Let $1<r \leq s < \infty$ and $a \in \C$, $0 < \mathfrak{Re}(a) <Q/r$. Let $0 \leq b <Q$ and $\frac{\mathfrak{Re}(a)}{Q} = \frac{1}{r} - \frac{1}{s} + \frac{b}{sQ}.$ Assume that $K:G \to \C$ is such that $|K_a(x)| \leq C_0 |x|_G^{ \mathfrak{Re}( a) - Q}$, for some positive constant $C_0$. Then there exists a positive constant $C_1=C_1(r,s,a, b)$ such that  $$\Big\| \frac{f * K_a}{|x|_G^{b/s}}  \Big\|_{L^s (G)} \leq C_1 \| f \|_{L^r (G)}.$$
\end{teo}
This particular integral inequality is special because left invariant operators on a graded $\K$-Lie group can be expressed as a right convolution operators. When the operator is homogeneous and satisfy some suitable conditions, the convolution kernel is bounded by $|x|_G^{a-Q}$ so that Theorem \ref{Teofunctionalineq} applies. For instance, we have the following examples: 
\begin{itemize}
    \item The Vladimirov-Taibleson operator on $G$ which is given by the convolution with the distribution that we call the Riesz kernel: 
\begin{align*}
    \langle \mathfrak{r}_s^G , f \rangle &:= \frac{1 - q^{-Q}}{1-q^{s-Q}} f(e) + \frac{1-q^{-s}}{1-q^{s-Q}} \int_{|x|_G >1} |x|_G^{s-Q} f(x) dx \\ &+ \frac{1-q^{-s}}{1-q^{s-Q}} \int_{|x|_G \leq 1} |x|_G^{s-Q} (f(x) - f(e))dx.
\end{align*} In particular for $\mathfrak{Re}(s)>0$ we can write: $$\langle \mathfrak{r}_s^G , f \rangle = \frac{1 - q^{-s}}{1 - q^{s-Q}} \int_G |x|_G^{s-Q} f(x)dx, \esp \esp \esp s \notin Q + \frac{2 \pi i}{\ln q} \Z,$$ $$\langle \mathfrak{r}_{-s}^G , f \rangle = \frac{1 - q^{s}}{1 - q^{-s-Q}} \int_G  |x|_G^{-s-Q} (f(x)- f(e))dx.$$The Vladimirov-Taibleson operator of order $a$, which we denote by $\mathscr{D}^a$ is given by the expression $$\mathscr{D}^a f(x) := \frac{1 - q^a}{1 - q^{- (a + Q)}} \int_G \frac{f(xy^{-1}) - f(x)}{|y|_G^{ a + Q}} dy,\esp \esp \esp \mathfrak{Re}(a)>0, \esp \esp x \in G.$$We can check that $\mathscr{D}^a$ is an $a$-homogeneous operator since \begin{align*}
    \mathscr{D}^a (f \circ D_\lambda) (x) &= \frac{1 - q^a}{1 - q^{- (a + Q)}} \int_G \frac{f(\lambda(xy)) - f(\lambda x)}{|y|_G^{ a + Q}} dy \\ &= \frac{1 - q^a}{1 - q^{- (a + Q)}} |\lambda|_\K^{a + Q} \int_G \frac{f(\lambda(xy)) - f(\lambda x)}{|\lambda y|_G^{ a + Q}} dy \\ &=\frac{1 - q^a}{1 - q^{- (a + Q)}} |\lambda|_\K^{a } \int_G \frac{f((\lambda x)y) - f( \lambda x)}{| y|_G^{ \alpha + Q}} dy   \\&= |\lambda|_\K^\alpha (\mathscr{D}^\alpha f) \circ D_\lambda (x). 
\end{align*}

The inverse operator $\mathscr{D}^{-a}$ is $$\mathscr{D}^{-a} f (x) = (f * | \cdot |_G^{a-Q} )(x) ,$$see \cite{FundSolVladimirovOp}, so that taking $\mathfrak{Re}(a) < Q$ and applying Theorem \ref{Teofunctionalineq} we get the following corollary: 
\begin{coro}
Let $G$ be a graded $\K$-Lie group with homogeneous dimension $Q$ and let $a \in \C$. Let $1<r \leq s < \infty$ and $0 < \mathfrak{Re}(a) <Q$. Let $0 \leq b <Q$ and $\frac{\mathfrak{ Re}(a)}{Q} = \frac{1}{r} - \frac{1}{s} + \frac{b}{sQ}.$ Then there exists a positive constant $C=C(r,s,a, b)$ such that  $$\Big\| \frac{f}{|x|^{b/s}_G}  \Big\|_{L^s (G)} \leq C \| \mathscr{D}^a f \|_{L^r (G)},$$for all $f \in \mathcal{D}(G).$
\end{coro}
         
\item Let $h $ be an elliptic homogeneous polynomial of homogeneous degree $\nu$ with coefficients in the ring of integers $\mathcal{O}_\K$. The Igusa's local zeta function associated to $h$ has a meromorphic continuation to the whole complex plane given by $$\langle |h|_\K^s , f \rangle = \int_{G_0} (f(x) - f(e)) |h(x)|_\K^s dx + \frac{L(q^{-s})}{1 - q^{-Q - \nu s}}f(e) + \int_{G \setminus G_0} f(x) |h(x)|_\K^s dx.$$In particular for $\mathfrak{Re}(s)>0$: $$ \langle |h|_\K^{s-Q/\nu} , f \rangle = \int_{G} f(x) |h(x)|_\K^{s-Q/\nu} dx, \esp \esp \esp  s \neq Q/\nu + \frac{2 \pi i}{\ln{q}} \Z, $$ $$ \langle |h|_\K^{-s-Q/\nu} , f \rangle = \int_{G} (f(x) - f(e)) |h(x)|_\K^{-s-Q/\nu} dx .$$The above distribution defines a convolution operator acting on $\mathcal{D}(G)$ by the formula $$T_{h,a} f (x)  := (f * |h|^{-a - Q/\nu}_\K)(x)  .$$See \cite{articlefundsolutionszuniga} for the relation between local zeta functions and fundamental solutions of pseudo-differential operators. As an application of Theorem \ref{Teofunctionalineq}, we obtain the following corollary:
\begin{coro}
Let $G$ be a graded $\K$-Lie group with homogeneous dimension $Q$ and let $a \in \C$. Let $h$ be an elliptic homogeneous polynomial of homogeneous degree $\nu>0$. Let $1<r \leq s < \infty$ and $0 < \mathfrak{Re}(a) <Q/\nu$. Let $0 \leq b <Q$ and $\frac{\mathfrak{ Re}(a) \nu}{Q} = \frac{1}{r} - \frac{1}{s} + \frac{b}{sQ}.$ Then there exists a positive constant $C=C(r,q,a, b)$ such that  $$\Big\| \frac{T_{h, - \alpha}f}{|x|^{b/s}_G}  \Big\|_{L^s (G)} \leq C \|  f \|_{L^r (G)},$$for all $f \in \mathcal{D}(G).$
\end{coro}
\item The Vladimirov Laplacian on a graded $\K$-Lie group is defined as the $a$-homogeneous left-invariant pseudo-differential operator given by $$\MO^a :=\sum_{k=1}^r \sum_{j=1}^{b_k} \partial_{X_{k,j}}^{a/\nu_k} ,$$where $\{ X_{k,j} \}$ is a Malcev basis associated to the gradation $$\mathfrak{g} = \bigoplus_{k=1}^{r} V_{\nu_k}, \esp \esp b_k := dim(V_{\nu_k}).$$Here we are using the notation introduced in the following definition:
\begin{defi}\normalfont
Let $G$ be a graded $\K$-Lie group with Lie algebra $\mathfrak{g}$. We define the \emph{directional VT operator} in the direction of $X \in \mathfrak{g}$ by the formula $$\partial_X^\alpha f (x) :=  \frac{1 - q^\alpha}{1 - q^{- (\alpha + 1)}} \int_{\K} \frac{f(x \cdot \mathbb{exp}(tX)^{-1}) - f(x)}{|t |_\K^{\alpha + 1}} dt .$$
\end{defi}
For the heat kernel of the Vladimirov Laplacian on the groups $\mathbb{H}_d$ or $\mathbb{E}_4$ the following properties are proven in \cite{FundSolVladimirovOp}:
\begin{teo}\label{teoheatkernelpropertiesgroups}
Let $G$ be either $\mathbb{H}_d$ or $\mathbb{E}_4$. Denote by $\MO^a_2$ the self-adjoint extension of $\MO^a$ to $L^2 (G)$. Then the heat kernel $h_{\MO^a_2} $ associated to the Vladimirov Laplacian $\MO^a$ on $G$ has the following properties: 

\begin{itemize}
    \item[(i)] $h_{\MO^a_2} (t, \cdot) *  h_{\MO^a_2} (s, \cdot) = h_{\MO^a_2} (t + s , \cdot )  $, for any $s,t>0$.
    \item[(ii)] $h_{\MO^a_2} ( | \gamma |_p^a t, D_\gamma (x)) = |\gamma|^{-Q}_p h_{\MO^a_2} ( t , x)$, for all $x \in G$ and any $t>0$, $\gamma \in \Q_p^*$.
\item[(iii)] $h_{\MO^a_2} ( t , x) = \overline{h_{\MO^a_2} ( t , x^{-1})},$ for all $x \in G$.  
\item[(iv)] The heat semigroup $e^{- t \MO_2^a}$ is symmetric and Markovian. Moreover, the following estimate holds for its kernel: $$h_{\MO^a_2} ( t , x) \asymp t^{-Q}.$$
\end{itemize}
\end{teo}
And as a corollary of the above it holds:
\begin{coro}
Let $G$ be either $\mathbb{H}_d$ or $\mathbb{E}_4$. Let $0 < a < Q.$ Then a fundamental solution for the Vladimirov Laplacian exists and it defines an $ a - Q$-homogeneous function given by $$\textbf{h}_{\MO^a_2}(x) := \int_0^{\infty} h_{\MO^a_2} ( t, x) dt.$$In consequence: $$\textbf{h}_{\MO^a_2}(x ) \asymp |x|_G^{a - Q}.$$ 
\end{coro}
\begin{coro}
Let $\beta\in \C$ such that $0< \mathfrak{Re}(\beta) < Q$. Then the linear operator $L^\beta:=(\MO_2^a)^{\beta/a}$ defined via functional calculus possesses a fundamental solution, which is an $\beta - Q$ homogeneous distribution, determined by the Riesz potential $$\mathcal{I}_\beta (x) := \frac{1}{\Gamma(\beta/a)} \int_0^\infty t^{\beta/a - 1} h_{\MO^a_2} ( t, x)dt.$$Furthermore, for $0<\beta<Q$ we have $$\mathcal{I}_\beta (x) \asymp |x|_G^{\beta - Q}.$$  
\end{coro}
See \cite{FundSolVladimirovOp} for all the details. Consequently, we get the following corollary of Theorem \ref{Teofunctionalineq}: 

\begin{coro}
Let $G$ be either $\mathbb{H}_d$ or $\mathbb{E}_4$ and denote by $Q$ the homogeneous dimension of $G$. Let $a \in \C$. Let $1<r \leq s < \infty$ and $0 < \mathfrak{Re}(a) <Q$. Let $0 \leq b <Q$ and $\frac{\mathfrak{ Re}(a)}{Q} = \frac{1}{r} - \frac{1}{s} + \frac{b}{sQ}.$ Then there exists a positive constant $C=C(r,s,a, b)$ such that  
\begin{equation}\label{HardySobolevin}
\left\| \frac{f}{|x|^{b/s}_G}  \right\|_{L^s (G)} \leq C \| \MO^a f \|_{L^r (G)} ,   
\end{equation}
for all $f \in \mathcal{D}(G).$

Moreover, if $b=s$ from \eqref{HardySobolevin}, we get the Hardy inequality and if $b=0$, from \eqref{HardySobolevin}, we get the Sobolev inequality. 
\end{coro}

\end{itemize}

Another consequence of Theorem \ref{Teofunctionalineq} is the following version of the uncertainly principle.
\begin{teo}[Uncertainly principle]\label{UP}
Let $G$ be either $\mathbb{H}_d$ or $\mathbb{E}_4$ and denote by $Q$ the homogeneous dimension of $G$. Let $a >0$, $r>1$ such that $Q>ar$. Then there exists a positive constant $C=C(r,a, Q)$ such that  
\begin{equation}\label{up}
\left\| f \right\|_{L^2 (G)} \leq C \| \MO^a f \|_{L^r (G)} \||\cdot|^{a}_Gf\|_{L^{r'}(G)},   
\end{equation}
for all $f \in \mathcal{D}(G).$
\end{teo}
In addition, we can prove as well the Gagliardo-Nirenberg inequality.
\begin{teo}[Gagliardo-Nirenberg inequality]\label{GNthm}
Let $G$ be either $\mathbb{H}_d$ or $\mathbb{E}_4$ and denote by $Q$ the homogeneous dimension of $G$. Let $a>0$, $r>1$, $Q>ar$, $\tau\geq 1$, $\alpha\in (0,1)$, $s>0$ and $\frac{1}{s}=\alpha\left(\frac{1}{r}-\frac{a}{Q}\right)+\frac{1-\alpha}{\tau}$. Then we have
\begin{equation}
    \left\| f \right\|_{L^s (G)}\leq C\| \MO^a f \|^{\alpha}_{L^r (G)}\|f\|^{1-\alpha}_{L^{\tau}(G)},
\end{equation}
where $C>0.$
\end{teo}

Finally, from Theorem \ref{Teofunctionalineq} we obtain the Hardy-Littlewood-Sobolev inequality on graded $\K$-Lie groups: 

\begin{teo}\label{teoHardy-Littlewood-Sobolev}
Let $G$ be a grade $\K$-Lie group of homogeneous dimension $Q$. Let $1<r<s< \infty$, $0<\lambda < Q,$ $\frac{1}{s}= \frac{1}{r} + \frac{\lambda}{Q} -1$, and $f \in L^r(G)$. Then we have $$\| f * | \cdot|^{-\lambda}_G \|_{L^s (G)} \lesssim \| f \|_{L^r (G)}. $$
\end{teo}

And from the Hardy-Littlewood-Sobolev inequality we obtain the Stein-Weiss inequality, proven in Theorem \ref{teosteinweiss}.

\begin{teo}\label{teosteinweiss}
Let $G$ be a graded $p$-adic Lie group of homogeneous dimension $Q$. Let $0< \lambda < Q$, $1< r < \infty$, $\beta  <Q/r'$, $\alpha  <Q/s$, $\alpha + \beta \geq 0$, $\frac{1}{s}= \frac{1}{r} + \frac{\alpha + \beta + \lambda}{Q} - 1$. Then for $1< r\leq s < \infty$, we have $$\| |\cdot|^{-\alpha}_GI_\lambda f \|_{L^s(G)} \leq C \| |\cdot|^\beta_G f \|_{L^r(G)}.$$
\end{teo}
\section{Strong Type Estimates}
To begin with the proof of our main results we need the following simple proposition about a change of variable on constant-order Vilenkin groups.

\begin{pro}\label{prointegralradialfunctions}
Let $G$ be a constant-order Vilenkin group, let us say $|G_n / G_{n+1}| = \varkappa$, for all $n \in \Z$. Then for every $f(| \cdot |) \in L^1_{loc} (G)$: $$\frac{1}{|x|} \int_{B(e,|x|)} f(|t|) dt = \int_{B(e,1)} f(|x| \cdot |t|) dt.$$
\end{pro}
\begin{proof}
The proof of this proposition is just a direct calculation. Take $x \in G_m \setminus G_{m+1}$ so that $|x|= \varkappa^{-m}$. Then \begin{align*}
    \frac{1}{|x|} \int_{B(e,|x|)} f(|t|) dt &= \varkappa^{m} \sum_{k \leq -m} f(\varkappa^k ) \int_{|t|=\varkappa^k} dt \\ &= \varkappa^{m} \sum_{k \leq -m} f(\varkappa^k ) \varkappa^{k} (1-  \varkappa^{-1}) \\ &= \varkappa^{m} \sum_{k \leq 0} f(\varkappa^{-m} \varkappa^k ) \varkappa^{-m} \varkappa^k (1-\varkappa^{-1}) \\ &= \sum_{k \leq 0} f(\varkappa^{-m} \varkappa^k ) \varkappa^k (1-\varkappa^{-1}) \\ &= \int_{B(e,1)} f(|x| \cdot |t|) dt.
\end{align*}This conclude the proof.  
\end{proof} 
The above change of variable, together with the H{\"o}lder inequality and the Minkowski integral inequality, are the necessary ingredients for the proof of Theorem \ref{teosharpestimateradialfunct}. 
\begin{proof}[Proof of Theorem \ref{teosharpestimateradialfunct} :]First we use Proposition \ref{prointegralradialfunctions}: 
\begin{align*}
    \| H_\delta f \|_{L^r_\alpha (G)} &= \Big(\int_G \Big| \frac{1}{|x|^{1 - \delta}} \int_{B(e, |x|) }f(|t|) dt \Big|^r |x|^\alpha dx \Big)^{1/r}\\ &= \Big( \int_G \Big| \int_{B(e , 1)} f(|x| \cdot |t|) dt \Big|^r |x|^{\alpha + \delta r} dx \Big)^{1/r}. 
\end{align*}Now we apply the Minkowski integral inequality: 
\begin{align*}
    \|H_\delta f\|_{L^r_\alpha (G)} &\leq \int_{B(e , 1)} \Big( \int_{G} |f(|x| \cdot |t|) |^r |x|^{\alpha + \delta r } dx \Big)^{1/r} dt \\ & =  \Big( \int_{B(e,1)} |t|^{-\frac{\alpha + \delta r + 1}{r}} dt \Big) \Big( \int_{G} |f(|x|) |^r |x|^{\alpha + \delta r} dx \Big)^{1/r} \\ &= (1- \varkappa^{-1})\Big( \sum_{k=0}^\infty \varkappa^{(\frac{\alpha}{r} - \frac{1}{r'} + \delta)k}  \Big)\|f \|_{L^r_{\alpha + \delta r} (G)} \\ &= \frac{1- \varkappa^{-1}}{1- \varkappa^{\frac{\alpha}{r} -\frac{1}{r'} + \delta}} \|f\|_{L^r_{\alpha + \delta r} (G)}.
\end{align*}
The above shows that $$\|H_\delta\|_{\mathcal{L}(\Tilde{L}^r_{\alpha + \delta r} (G) ,\Tilde{L}^r_\alpha (G))} \leq \frac{1- \varkappa^{-1}}{1- \varkappa^{\frac{\alpha}{r} -\frac{1}{r'}}}.$$For the converse inequality we use a similar trick to the one used by Hardy to prove the sharpness of the constant $r/(r-1)$ in the original Hardy inequality, see \cite{10.2307/27642033}. Take a sequence $\{ g_n\}_{n \in \N_0}$ in $G$ such that $|g_n|=\varkappa^n$ and define $\varepsilon_n := \varkappa^{-n}$. Consider the functions $f_n$ given by \[f_n (x) :=\begin{cases} 0 & \esp \esp \text{if} \esp |x| < 1, \\ |x|^{-\frac{\alpha + \delta r + 1}{r} - \varepsilon_n}  & \esp \esp \text{if} \esp |x| \geq 1.\end{cases}\]Clearly $$\| f_n \|_{L^r_{\alpha + \delta r} (G)}^r = \frac{1-\varkappa^{-1}}{1 - \varkappa^{- \varepsilon_n r}},$$ and \[H_\delta f_n (x) =\begin{cases} 0 & \esp \esp \text{if} \esp |x| < 1 , \\ |x|^{-\frac{\alpha + 1}{r} - \varepsilon_n} \int_{\frac{1}{|x|} \leq |t| \leq 1} |t|^{-\frac{\alpha + \delta r  + 1}{r} - \varepsilon_n}  & \esp \esp \text{if} \esp |x| \geq 1.\end{cases}\]We can estimate \begin{align*}
    \| H_\delta f \|_{L^r_\alpha(G)} &= \Big(\int_{|x| \geq 1} \Big( |x|^{-\frac{\alpha +1}{r} - \varepsilon_n} \int_{\frac{1}{|x|} \leq |t| \leq 1} |t|^{-\frac{\alpha + \delta r +1}{r} - \varepsilon_n} dt\Big)^r |x|^\alpha dx \Big)^{1/r} \\ & \geq \Big(\int_{|x| \geq |g_n|} \Big( |x|^{-\frac{\alpha +1}{r} - \varepsilon_n} \int_{\frac{1}{|g_n|} \leq |t| \leq 1} |t|^{-\frac{\alpha + \delta r +1}{r} - \varepsilon_n} dt\Big)^r |x|^\alpha dx\Big)^{1/r} \\&= \Big( \int_{|x| \geq |g_n|} |x|^{-1 - r \varepsilon_n} dx \Big)^{1/r} \int_{\frac{1}{|g_n|} \leq |t| \leq 1} |t|^{-\frac{\alpha + \delta r +1}{r} - \varepsilon_n} dt \\ &= |g_n|^{-\varepsilon_n} \|f_n\|_{L^r_{\alpha + \delta r} (G)} \int_{\frac{1}{|g_n|} \leq |t| \leq 1} |t|^{-\frac{\alpha + \delta r +1}{r} - \varepsilon_n} dt,
\end{align*}which shows that $$\|H_\delta \|_{\mathcal{L} (L^r_{\alpha + \delta r} (G) , L^r_\alpha (G))} \geq  |g_n|^{-\varepsilon_n} \int_{\frac{1}{|g_n|} \leq |t| \leq 1} |t|^{-\frac{\alpha + \delta r +1}{r} - \varepsilon_n} dt.$$Finally, we take the limit as $n \to \infty$ to conclude $$\|H_\delta \|_{\mathcal{L} (L^r_{\alpha + \delta r} (G) , L^r_\alpha (G))} \geq  \int_{ 0< |t| \leq 1} |t|^{-\frac{\alpha + \delta r +1}{r} } dt = \frac{1- \varkappa^{-1}}{1- \varkappa^{\frac{\alpha}{r} -\frac{1}{r'} + \delta }}.$$This finish the proof.
\end{proof}
As the reader may notice the proof of the above result is elementary and it only requires H{\"o}lder inequality, Minkowski integral inequality and a change of variable. The reason why we estimate first the norm of the Hardy operator on radial functions is that in general we don't have a nice change of variable on Vilenkin groups. That's why we need to consider constant-order Vilenkin groups, because when the order of the quotients $|G_n/G_{n+1}|$ is constant we have Proposition \ref{prointegralradialfunctions}. If the group is not constant-order or if we consider non-radial functions the situation is more complicated. Anyway, for graded $\K$-Lie groups we can perform a change of variable that will allows to show that the norm of the Hardy operator acting only on radial functions is the same as when acting on the whole space $L^r_\alpha (G)$. This situation rise the natural question about whether or not the same holds true for more general classes of Vilenkin groups.  

\begin{rem}
In contrast with the real case, in the $p$-adic case the norm of the Hardy operator depends on the dimension of the group. See \cite{Zun} where the case $G= \Q_p^d$ is treated in detail. 
\end{rem}

\begin{rem}
We can easily check that for graded $\K$-Lie groups it holds $$|D_\lambda (S)|= | \lambda|_\K^{Q}  |S| , \esp \esp \int_{G} f(\lambda x) dx = |\lambda|_\K^{-Q} \int_G f(x)dx.$$ As a consequence of the above, we have on $G$ the following analogous of the polar decomposition: 
\begin{align*} 
    \int_{G} f(x) dx &= \sum_{k \in \Z} \int_{G_k \setminus G_{k+1}} f(x) dx \\ &=\sum_{n \in \Z} \int_{\mathscr{p}^k G_0 \setminus \mathscr{p}^k G_1} f(x)dx \\ &= \sum_{k \in \Z} q^{-k Q} \int_{G_0 \setminus G_1} f(\mathscr{p}^k x) dx.
\end{align*}
Also, we will use the notation $$ord_G(x) := - \log_p |x|_G  , \esp \esp x \in G \setminus \{ e \},$$and $\lambda (x) := \mathscr{p}^{ord_\mathscr{G} (x)} .$ In this way we can write $$\int_{|t|_\mathscr{G} \leq |x|_{\mathscr{G}}} f(t) dt = |\lambda (x)|_\K^Q \int_{| t|_{\mathscr{G}} \leq 1} f ( \lambda (x) t)  dt = |x|_G^Q \int_{| t|_{\mathscr{G}} \leq 1} f ( \lambda (x) t)  dt. $$
\end{rem}

\begin{proof}[Proof Corollary \ref{corostronghardyhomopadicLie}:]
Given a function $f$, we can associate to it a radial function $g$ such that $Hf = Hg$, just set $$g(x):=\frac{1}{1 - p^{-\emph{Q}}} \int_{|t| =1} f(\lambda (x) t) dt, \esp \esp x \in G.$$Then clearly $g$ is a radial function and $H_\delta g = H_\delta f$. Applying the H{\"o}lder inequality we get  \begin{align*}
    \|g\|_{L^r_{\alpha + \delta r} (G)} &= \Big( \int_G \Big|\frac{1}{1-p^{-\emph{Q}}} \int_{|t|=1} f(\lambda(x) t) dt \Big|^r |x|^{\alpha + \delta r} dx \Big)^{1/r} \\ & \leq \Big(\int_G \frac{1}{(1-p^{-\emph{Q}})^r} \Big( \int_{|t|=1} |f(\lambda(x) t)|^r dt\Big) \Big( \int_{|t|=1} dt \Big)^{r/r'} |x|^{\alpha + \delta r} dx \Big)^{1/r} \\ &= \Big( \int_G \frac{1}{1-p^{-\emph{Q}}} \Big(\int_{|t| =1} |f( \lambda(x) t)|^r dt \Big) |x|^{\alpha + \delta r} dx \Big)^{1/r} \\ & = \frac{1}{(1- p^{-\emph{Q}})^{1/r}} \Big( \int_G \int_{|x|  = |y| } |f(y)|^r   |x|^{\alpha + \delta r - 1} dy dx \Big)^{1/r} \\ &= \| f \|_{L^r_{\alpha + \delta r} (G)}.
\end{align*}Therefore we obtain $$\frac{\| H_\delta f\|_{L^r_{\alpha } (G)}}{\| f \|_{L^r_{\alpha + \delta r} (G)}} \leq \frac{\| H_\delta g\|_{L^r_\alpha (G)}}{\| g \|_{L^r_{\alpha + \delta r} (G)}},$$which implies the equality between the norm of the Hardy operator in $\mathcal{L} (L^r_{\alpha + \delta r } (G) , L^r_\alpha (G))$ and the norm of its restriction to radial functions. The sharp estimate on the norm of the Hardy operator on radial functions given in Theorem \ref{teosharpestimateradialfunct} conclude the proof. 
\end{proof}

Now we proceed with the proof of the sharp estimate on the Hardy-Littlewood-P{\'o}lya operator:

\begin{proof}[Proof of Theorem \ref{TeoHLPradialfunctions}:]
Let us take a radial function $f \in L^r_\alpha (G)$. First we need to perform a suitable change of variable and to apply the Minkowski integral inequality:

\begin{align*}
    \| T f\|_{L^r_\alpha (G)} & = \Big(\int_G \Big| \int_{G \setminus \{ e\}} \frac{f(|y|)}{\max\{|x| , |y| \}} dy \Big|^r |x|^\alpha dx \Big)^{1/r} \\ &= \Big(\int_G \Big| \int_{G \setminus \{ e\}} \frac{f(|x| \cdot |y|)}{\max\{1 , |y| \}} dy \Big|^r |x|^\alpha dx \Big)^{1/r} \\ & \leq \int_{G \setminus \{ e\}} \Big( \int_{G} \frac{|f(|x|\cdot|y|)|^r}{\max\{ 1, |y|^r\}} |x|^\alpha dx \Big)^{1/r} dy \\ &= \int_{G \setminus \{ e \}} \frac{1}{\max \{ 1 , |y| \} } \Big( \int_{G} |f(|x|\cdot|y|)|^r |x|^\alpha dx \Big)^{1/r} dy  \\ &= \int_{G \setminus \{ e \}} \frac{|y|^{-\frac{\alpha + 1}{r}}}{\max \{ 1 , |y| \} } dy \Big( \int_{G} |f(|x|)|^r |x|^\alpha dx \Big)^{1/r}. 
\end{align*}Now we just have to calculate \begin{align*}
    \int_{G \setminus \{ e \}} \frac{|y|^{-\frac{\alpha + 1}{r}}}{\max \{ 1 , |y| \} } dy & = \sum_{k=0}^\infty \varkappa^{-k(1-\frac{\alpha + 1}{r})} (1- \varkappa^{-1}) + \sum_{k=1}^\infty \varkappa^{-k(\frac{\alpha + 1}{r})} (1 - \varkappa^{-1}) \\ &=(1-\varkappa^{-1})\Big(\frac{1}{1-\varkappa^{\frac{\alpha}{r} - \frac{1}{r'}}} + \frac{\varkappa^{-\frac{\alpha + 1}{r}}}{1-\varkappa^{-\frac{\alpha+1}{r}}} \Big).
\end{align*}Summing up we get $$\| T\|_{\mathcal{L}(\Tilde{L}^r_\alpha (G))} \leq (1-\varkappa^{-1})\Big(\frac{1}{1-\varkappa^{\frac{\alpha}{r} - \frac{1}{r'}}} + \frac{\varkappa^{-\frac{\alpha + 1}{r}}}{1-\varkappa^{-\frac{\alpha+1}{r}}} \Big).$$
Now we want to show that the above estimate is sharp. For doing this we will use the same trick as in the proof of Theorem \ref{teosharpestimateradialfunct}. Take a sequence $\{ g_n\}_{n \in \N_0}$ in $G$ such that $|g_n|=\varkappa^n$ and define $\varepsilon_n := \varkappa^{-n}$. Consider the functions $f_n$ given by \[f_n (x) :=\begin{cases} 0 & \esp \esp \text{if} \esp |x| < 1, \\ |x|^{-\frac{\alpha + 1}{r} - \varepsilon_n}  & \esp \esp \text{if} \esp |x| \geq 1.\end{cases}\]Clearly again we have $$\| f_n \|_{L^r_\alpha (G)}^r = \frac{1-\varkappa^{-1}}{1 - \varkappa^{- \varepsilon_n r}},$$ and we can obtain $$T f_n (x) = \int_{|y| \geq 1} \frac{|y|^{- \frac{\alpha + 1}{r} - \varepsilon_n}}{\max\{|x| , |y|\}} dy .$$ Now we can estimate below the norm of $T f_n$ as follows: \begin{align*}
    \| T f_n \|_{L^r_\alpha (G)} &= \Big(\int_G \Big| \int_{|y| \geq 1} \frac{|y|^{- \frac{\alpha + 1}{r} - \varepsilon_n}}{\max\{|x| , |y|\}} dy \Big|^r |x|^\alpha dx \Big)^{1/r} \\ & \geq \Big(\int_{|x| \geq |g_n|} \Big| \int_{|t| \geq \frac{1}{|g_n|}} \frac{|t|^{- \frac{\alpha + 1}{r} - \varepsilon_n}}{\max\{|1 , |t|\}} dt \Big|^r |x|^{-1 - \varepsilon_n r} dx \Big)^{1/r} \\ & = \| f_n \|_{L^r_\alpha (G)} |g_n|^{- \varepsilon_n} \int_{|t| \geq \frac{1}{|g_n|}} \frac{|t|^{- \frac{\alpha + 1}{r} - \varepsilon_n}}{\max\{|1 , |t|\}} dt.
\end{align*}The above clearly implies $$\|T\|_{\mathcal{L}(\Tilde{L}^r_\alpha (G))} \geq |g_n|^{- \varepsilon_n} \int_{|t| \geq \frac{1}{|g_n|}} \frac{|t|^{- \frac{\alpha + 1}{r} - \varepsilon_n}}{\max\{|1 , |t|\}} dt.$$Taking the limit as $n \to \infty $we conclude the proof.
\end{proof}
In the same way as for the Hardy operator, for the graded $\K$-Lie groups we can show that the norm of the Hardy-Littlewood-P{\'o}lya acting on $L^r_\alpha (G)$ is the same as the norm of its restriction to radial functions. The trick for doing so is, once again, a simple change of variable that might be not possible in general for Vilenkin groups.

\begin{proof}[Proof Corollary \ref{coroHLPp-adicLiehomogeneousn}:]
First $$\|T\|_{\mathcal{L}(L^r_\alpha (G))} \geq \|T\|_{\mathcal{L}(\Tilde{L}^r_\alpha (G))} =  (1-p^{-\normalfont{\emph{Q}}})\Big(\frac{1}{1-p^{\normalfont{\emph{Q}}(\frac{\alpha}{r} - \frac{1}{r'})}} + \frac{p^{-\normalfont{\emph{Q}}\frac{\alpha + 1}{r}}}{1-p^{-\normalfont{\emph{Q}}\frac{\alpha+1}{r}}} \Big),$$as a consequence of Theorem \ref{TeoHLPradialfunctions}. For the converse inequality we have: \begin{align*}
    \| T f\|_{L^r_\alpha (G)} & = \Big(\int_G \Big| \int_{G \setminus \{ e\}} \frac{f(y)}{\max\{|x| , |y| \}} dy \Big|^r |x|^\alpha dx \Big)^{1/r} \\ &= \Big(\int_G \Big| \int_{G \setminus \{ e\}} \frac{f(\lambda(x) y)}{\max\{1 , |y| \}} dy \Big|^r |x|^\alpha dx \Big)^{1/r} \\ & \leq \int_{G \setminus \{ e\}} \Big( \int_{G} \frac{|f(\lambda(x) y)|^r}{\max\{ 1, |y|^r\}} |x|^\alpha dx \Big)^{1/r} dy \\ &= \int_{G \setminus \{ e \}} \frac{1}{\max \{ 1 , |y| \} } \Big( \int_{G} |f(\lambda(x) y)|^r |x|^\alpha dx \Big)^{1/r} dy  \\ &\leq  \int_{G \setminus \{ e \}} \frac{|y|^{-\frac{\alpha + 1}{r}}}{\max \{ 1 , |y| \} } dy \Big( \int_{G} |f(x)|^r |x|^\alpha dx \Big)^{1/r} \\ &=  (1-p^{-\normalfont{\emph{Q}}})\Big(\frac{1}{1-p^{\normalfont{\emph{Q}}(\frac{\alpha}{r} - \frac{1}{r'})}} + \frac{p^{-\normalfont{\emph{Q}}\frac{\alpha + 1}{r}}}{1-p^{-\normalfont{\emph{Q}}\frac{\alpha+1}{r}}} \Big) \| f \|_{L^r_\alpha (G)}.
\end{align*}This conclude the proof.
\end{proof}

\section{Weak Type Estimates}

\begin{proof}[Proof of Theorem \ref{teoweakestHardy}:]
Applying the H{\"o}lder inequality we get the estimate \begin{align*}
    | H_\delta f (x) | &= \Big| \frac{1}{|x|^{1 - \delta}} \int_{|y| \leq |x|} f(y) dy \Big| \\ & \leq |x|^{\delta - 1} \Big( \int_{|y| \leq |x|} |f(y)|^r |y|^\beta dy \Big)^{1/r} \Big( \int_{|y|\leq |x|} |y|^{- \beta r' / r} dy \Big)^{1/r'} \\ & \leq |x|^{\delta - 1} \| f \|_{L^r_\beta (G)} \Big( \int_{|y|\leq |x|} |y|^{- \beta r' / r} dy \Big)^{1/r'}.
\end{align*}Now a simple calculation yields to $$\Big( \int_{|y|\leq |x|} |y|^{- \beta r' / r} dy \Big)^{1/r'} = \Big( \frac{1 - \varkappa^{-1}}{1 - \varkappa^{\frac{\beta}{r-1} -1 }} \Big)^{1/r'} |x|^{\frac{1}{r'} - \frac{\beta}{r}},$$so, recalling the condition $\frac{\beta + 1}{r} - \delta = \frac{\gamma+1}{s}$, we obtain the inequality $$|H_\delta f (x)| \leq \Big( \frac{1 - \varkappa^{-1}}{1 - \varkappa^{\frac{\beta}{r-1} -1 }} \Big)^{1/r'} \| f \|_{L^r_\beta (G)} |x|^{- \frac{\gamma + 1}{s}}.$$For simplicity let us define now the constant $$C_f :=\Big( \frac{1 - \varkappa^{-1}}{1 - \varkappa^{\frac{\beta}{r-1} -1 }} \Big)^{1/r'} \| f \|_{L^r_\beta (G)} .$$Then for any $\lambda >0$ it holds $$A_\lambda := \{x \in G \esp : \esp |H_\delta f (x)| > \lambda \} \subset B_\lambda := \Big\{x \in G \esp: \esp |x| < \Big(\frac{C_f}{\lambda}\Big)^{\frac{s}{\gamma + 1}} \Big\},$$and by using the inclusion $A_\lambda \subset B_\lambda$ we can estimate \begin{align*}
    \| H_\delta f \|_{L^{s, \infty}_\gamma (G)} &= \sup_{\lambda >0} \lambda \cdot \Big( \int_{A_\lambda} |x|^\gamma dx \Big)^{1/s} \\ & \leq \sup_{\lambda >0} \lambda \cdot \Big( \int_{B_\lambda} |x|^\gamma dx \Big)^{1/s}.
\end{align*}Let us denote by $k_\lambda$ the unique integer number such that $$\varkappa^{k_\lambda} \leq \Big(\frac{C_f}{\lambda}\Big)^{\frac{s}{\gamma + 1}}< \varkappa^{k_\lambda +1}.$$Then \begin{align*}
    \int_{B_\lambda} |x|^\gamma dx & = (1 - \varkappa^{-1}) \sum_{k \leq k_\lambda} \varkappa^{k(\gamma + 1)} \\ &= (1 - \varkappa^{-1}) \varkappa^{k_\lambda (\gamma + 1)} \sum_{k \leq 0} \varkappa^{k(\gamma + 1)} \\ & \leq (1 - \varkappa^{-1}) \Big( \frac{C_f}{\lambda} \Big)^s \sum_{k \leq 0} \varkappa^{k(\gamma + 1)} \\ &= (1 - \varkappa^{-1}) \Big( \frac{C_f}{\lambda} \Big)^s \Big( \frac{1}{1 - \varkappa^{- (\gamma + 1)}} \Big),
\end{align*}so it follows \begin{align*}
    \|H_\delta f \|_{L^{s , \infty}_\gamma (G)} \leq \Big( \frac{1 - \varkappa^{-1}}{1 - \varkappa^{-(\gamma + 1)}} \Big)^{1/s} C_f = \Big( \frac{1 - \varkappa^{-1}}{1 - \varkappa^{-(\gamma + 1)}} \Big)^{1/s} \Big( \frac{1 - \varkappa^{-1}}{1 - \varkappa^{\frac{\beta}{r-1} -1 }} \Big)^{1/r'} \| f \|_{L^r_\beta (G)} .
\end{align*}This shows that $$\| H_\delta \|_{\mathcal{L}(L^r_\beta (G) , L^{s , \infty}_\gamma (G))} \leq \Big( \frac{1 - \varkappa^{-1}}{1 - \varkappa^{-(\gamma + 1)}} \Big)^{1/s} \Big( \frac{1 - \varkappa^{-1}}{1 - \varkappa^{\frac{\beta}{r-1} -1 }} \Big)^{1/r'}.$$Now, in order to show that the above estimate is sharp, let us define the function $f_0$ by $$f_0 (x) = |x|^{\frac{- \beta}{r-1}} \mathbb{1}_{G_0} (x).$$Hence by using the identities $r r' = r + r'$ and $(r-1)(r'-1) = 1$ we get \begin{align*}
    \| f_0 \|_{L^r_\beta (G)} &= \Big( \int_{G_0} |x|^{- \beta( \frac{r}{r-1} - 1)} dx \Big)^{1/r} \\&= \Big( \int_{G_0} |x|^{- \beta( r' - 1)} dx \Big)^{1/r} \\ &= \Big( \int_{G_0} |x|^{- \frac{\beta}{r-1}} dx \Big)^{1/r}   \\& = \Big( \frac{1 - \varkappa^{-1}}{1 - \varkappa^{\frac{\beta}{r-1} - 1}} \Big)^{1/r}
\end{align*} Let us apply now the operator $H_\delta$ to the function $f_0$. We obtain: \[H_\delta f_0 (x)= |x|^{\delta - 1} \begin{cases} \int_{|y| \leq |x|} |y|^{- \frac{\beta}{r-1}} dy & \esp \esp \text{if} \esp |x| \leq 1, \\ \int_{G_0} |y|^{-\frac{ \beta}{r-1}} dy  & \esp \esp \text{if} \esp |x| > 1 .\end{cases}\]By calculating the above integrals we obtain: \[H_\delta f_0 (x)= \frac{1 - \varkappa^{-1}}{1 - \varkappa^{\frac{\beta}{r-1} - 1}} \begin{cases} |x|^{-(\frac{\beta}{r-1} - \delta)} & \esp \esp \text{if} \esp |x| \leq 1, \\ |x|^{- (1 - \delta)} & \esp \esp \text{if} \esp |x| > 1 .\end{cases}\]Let us define now the constant $C_0$ by $$C_0 := \frac{1 - \varkappa^{-1}}{1 - \varkappa^{\frac{\beta}{r-1} - 1}} = \Big(\frac{1 - \varkappa^{-1}}{1 - \varkappa^{\beta - 1}}\Big)^{1/r'} \| f_0 \|_{L^r_\beta (G)},$$and for $\lambda>0$ let us define the sets $$A_\lambda^1:= \Big\{ x \in G_0 \esp : \esp |x|< \Big( \frac{C_0}{\lambda} \Big)^{\frac{1}{\frac{\beta}{r-1} - \delta}} \Big\}, \esp \esp \esp \esp A_\lambda^2 := \Big\{ x \in G \setminus G_0 \esp : \esp |x|< \Big( \frac{C_0}{\lambda} \Big)^{\frac{1}{1 - \delta}} \Big\}$$and $$A_\lambda := \{x \in G \esp : \esp |H_\delta f_0 (x)|> \lambda\} = A_\lambda^1 \cup A_\lambda^2.$$The weak $L^s_\gamma$-norm of $H_\delta f_0$ can be written as \begin{align*}
    \| H_\delta f_0 \|_{L^{s, \infty}_\gamma (G)} &= \sup_{\lambda >0} \lambda \cdot \Big( \int_{A_\lambda} |x|^\gamma dx \Big)^{1/s} \\ &= \max \Big\{ \sup_{C_0 > \lambda >0} \lambda \cdot \Big( \int_{A_\lambda} |x|^\gamma dx \Big)^{1/s}, \esp \sup_{\lambda \geq C_0} \lambda \cdot \Big( \int_{A_\lambda} |x|^\gamma dx \Big)^{1/s} \Big\},
\end{align*}and clearly $$A_\lambda = G_0 \cup A_\lambda^2 =\Big\{ x \in G  \esp : \esp |x|< \Big( \frac{C_0}{\lambda} \Big)^{\frac{1}{1 - \delta}} \Big\}, \esp \esp \text{when} \esp \esp 0< \lambda < C_0 ,$$ $$A_\lambda = A_\lambda^1 \esp \esp \text{when} \esp \esp C_0 \leq \lambda.$$In one hand \begin{align*}
\sup_{C_0 > \lambda >0} \lambda \cdot \Big( \int_{A_\lambda} |x|^\gamma dx \Big)^{1/s} & =\Big( \frac{1 - \varkappa^{-1}}{1 - \varkappa^{-(\gamma + 1)}} \Big)^{1/s} \sup_{0 < \lambda < C_0} \lambda \cdot \Big( \frac{C_0}{\lambda} \Big)^{\frac{\gamma + 1}{s(1 - \delta)}} \\ & \leq \Big( \frac{1 - \varkappa^{-1}}{1 - \varkappa^{-(\gamma + 1)}} \Big)^{1/s} C_0 \\ &= \Big( \frac{1 - \varkappa^{-1}}{1 - \varkappa^{-(\gamma + 1)}} \Big)^{1/s} \Big(\frac{1 - \varkappa^{-1}}{1 - \varkappa^{\frac{\beta}{r-1} - 1}}\Big)^{1/r'} \| f_0 \|_{L^r_\beta (G)},    
\end{align*}where the above inequality holds because $C_0 / \lambda >1$ and $(\gamma + 1) /s(1 - \delta) \leq 1$. On the other hand \begin{align*}
    \sup_{\lambda \geq C_0} \lambda \cdot \Big( \int_{A_\lambda} |x|^\gamma dx \Big)^{1/s} &= \sup_{\lambda \geq C_0} \lambda \cdot \Big( \int_{A_\lambda^1} |x|^\gamma dx \Big)^{1/s} \\ &= \Big( \frac{1 - \varkappa^{-1}}{1 - \varkappa^{-(\gamma + 1)}} \Big)^{1/s} \sup_{\lambda \geq C_0} \lambda \cdot  \Big( \frac{C_0}{\lambda} \Big)^{\frac{\gamma + 1}{s(\frac{\beta}{r-1} - \delta)}} \\ &=  \Big( \frac{1 - \varkappa^{-1}}{1 - \varkappa^{-(\gamma + 1)}} \Big)^{1/s} C_0 \\ &= \Big( \frac{1 - \varkappa^{-1}}{1 - \varkappa^{-(\gamma + 1)}} \Big)^{1/s} \Big(\frac{1 - \varkappa^{-1}}{1 - \varkappa^{\frac{\beta}{r-1} - 1}}\Big)^{1/r'} \| f_0 \|_{L^r_\beta (G)},
\end{align*}where the above holds because the conditions $\beta < r-1$ and $\frac{\beta + 1}{r} - \delta = \frac{\gamma + 1}{s}$ yield $\frac{\gamma + 1}{s}>\frac{\beta }{r-1} - \delta.$ In conclusion we obtain $$\| H_\delta \|_{\mathcal{L} (L^r_\beta (G), L^{s, \infty}_\gamma (G))} = \Big( \frac{1 - \varkappa^{-1}}{1 - \varkappa^{-(\gamma + 1)}} \Big)^{1/s} \Big(\frac{1 - \varkappa^{-1}}{1 - \varkappa^{\frac{\beta}{r-1} - 1}}\Big)^{1/r'}.$$
\end{proof}
With very similar arguments we can prove the case $r=1$:
\begin{proof}[Proof of Theorem \ref{teoweakHardyL1}:]
As before we can estimate \begin{align*}
    |H_\delta f (x) | & \leq |x|^{\delta - 1} \int_{|y| \leq |x|} |f(y)| dy \\ &= |x|^{\delta - 1} \int_{|y| \leq |x|} |f(y)| |y|^{-\beta} |y|^{\beta} dy \\ & \leq  |x|^{-(1- \delta - \beta )} \int_{|y| \leq |x|} |f(y)| |y|^{-\beta} dy \\ & \leq |x|^{-(1- \delta - \beta )} \| f \|_{L^1_\beta (G)}, 
\end{align*} so, defining $C_f :=\| f \|_{L^1_\beta (G)}$, we get $$A_\lambda := \{x \in G \esp : \esp |H_\delta f (x)| >\lambda \} \subset B_\lambda := \Big\{ x \in G \esp : \esp |x|< \Big( \frac{C_f}{\lambda} \Big)^{\frac{1}{1- \delta - \beta}} \Big\}.$$It follows for $\gamma + 1 = s(1 - \delta - \beta)$ \begin{align*}
    \| H_\delta f \|_{L^{s , \infty}_\gamma (G)} &= \sup_{\lambda>0} \lambda \cdot \Big(   \int_{A_\lambda} |x|^\gamma dx \Big)^{1/s} \\ & \leq \sup_{\lambda>0} \lambda \cdot \Big( \int_{B_\lambda} |x|^\gamma dx \Big)^{1/s} \\ &\leq \Big( \frac{1 - \varkappa^{-1}}{1 - \varkappa^{-(\gamma + 1)}} \Big)^{1/s} \| f \|_{L^1_\beta (G)} \\ &= \Big( \frac{1 - \varkappa^{-1}}{1 - \varkappa^{-s(1- \delta - \beta)}} \Big)^{1/s} \| f \|_{L^1_\beta (G)}.  
\end{align*}Now let us show that the above estimate is sharp when $\beta = 0$. For doing this define $f_0(x):= \mathbb{1}_{G_0} (x)$. Then $\| f_0 \|_{L^1(G)} =1$ and \[H_\delta f_0 (x)=  \begin{cases} |x|^{\delta } & \esp \esp \text{if} \esp |x| \leq 1, \\ |x|^{- (1 - \delta)} & \esp \esp \text{if} \esp |x| > 1 .\end{cases}\]Write again $$A_\lambda := \{x \in G \esp : \esp |H_\delta f_0 (x)|>\lambda\}.$$Then clearly $A_\lambda = \emptyset$ when $\lambda \geq 1,$ and for $0 < \lambda < 1$ $$A_\lambda = \{x \in G \esp : \esp \lambda^{\frac{1}{\delta}}< |x|< (1/\lambda)^{\frac{1}{1- \delta}} \}.$$Let us define now $k^1_\lambda, k^2_\lambda$ as the unique integer numbers such that $$\varkappa^{k_\lambda^1} \leq \lambda^{\frac{1}{\delta}} < \varkappa^{k_\lambda^1 + 1}, \esp \esp \text{and} \esp \esp \varkappa^{k_\lambda^2 } \leq (1/\lambda)^{\frac{1}{1- \delta}} < \varkappa^{k_\lambda^2 + 1}.$$In this way \begin{align*}
    \| H_\delta f_0 \|_{L^{s , \infty}_\gamma (G)} &= \sup_{0 < \lambda < 1} \lambda \cdot \Big( \int_{A_\lambda} |x|^\gamma dx \Big)^{1/s} \\ &= (1 - \varkappa^{-1})^{1/s} \sup_{0 < \lambda < 1} \lambda \cdot \Big( \sum_{k = k_\lambda^1 + 1}^{k_\lambda^2 } \varkappa^{k(\gamma + 1)} \Big)^{1/s} \\ &= (1 - \varkappa^{-1})^{1/s} \sup_{0 < \lambda < 1} \lambda \cdot \Big( \frac{\varkappa^{(k_\lambda^1 +1)(\gamma + 1)} - \varkappa^{(k_\lambda^2 + 1)(\gamma + 1)}}{1 - \varkappa^{\gamma + 1}}  \Big)^{1/s} \\ & = (1 - \varkappa^{-1})^{1/s} \sup_{0<\lambda < 1} \lambda \cdot \Big( \frac{   \lambda^{\frac{\gamma + 1}{ \delta - 1}} - \lambda^{\frac{\gamma + 1}{\delta}}}{1- \varkappa^{-(\gamma + 1)}} \Big)^{1/s} \\ & = \Big( \frac{1 - \varkappa^{-1}}{1 - \varkappa^{- (\gamma + 1)}}  \Big)^{1/s} \Big( \sup_{0 < \lambda < 1} 1 - \lambda^{\frac{\gamma + 1}{1-\delta} } \lambda^{\frac{\gamma + 1}{\delta}}\Big)^{1/s} = \Big( \frac{1 - \varkappa^{-1}}{1 - \varkappa^{- s(1 - \delta)}}  \Big)^{1/s}.
\end{align*}This concludes the proof.
\end{proof}

\section{The Adjoint Operator}
To conclude with the fist part of our work we include estimates for the adjoint Hardy operator analogous to the ones given for the Hardy operator. As the reader will notice the proof of the estimates for the adjoint operator will be in the same lines as the proof for the fractional Hardy operator.   

\begin{pro}\label{prochangeofvariable2}
Let $G$ be a constant-order Vilenkin group, let us say $|G_n / G_{n+1}| = \varkappa$ for all $n \in \Z$. Then for a radial function $f(|\cdot|)$ it holds: $$\frac{1}{|x|} \int_{|y|>|x|} f(|y|) dy = \int_{|y|>1} f(|x| \cdot |y|) dy .$$
\end{pro}

\begin{proof}[Proof of Theorem \ref{teosharpstrongfractionaladjointHardyradial}:]
The proof is very similar to the proof of Theorem \ref{teosharpestimateradialfunct}. 
\begin{align*}
    \| H_\delta^* f \|_{L^r_\alpha (G)} &= \Big(\int_G \Big|  \int_{|t|> |x| } \frac{f(|t|)}{|t|^{1 - \delta}} dt \Big|^r |x|^\alpha dx \Big)^{1/r}\\ &= \Big( \int_G \Big| \int_{|t|>1} \frac{f(|x| \cdot |t|)}{|t|^{1 - \delta}} dt \Big|^r |x|^{\alpha + \delta r} dx \Big)^{1/r} \\ &\leq \int_{|t|> 1} |t|^{-(1 - \delta)} \Big( \int_{G} |f(|x| \cdot |t|) |^r |x|^{\alpha + \delta r } dx \Big)^{1/r} dt \\ & =  \Big( \int_{|t|>1} |t|^{-(1 - \delta) - \frac{\alpha + \delta r + 1}{r}} dt \Big) \Big( \int_{G} |f(|x|) |^r |x|^{\alpha + \delta r} dx \Big)^{1/r}\\ & =  \Big( \int_{|t|>1} |t|^{ -1 - \frac{\alpha +1 }{r}} dt \Big) \Big( \int_{G} |f(|x|) |^r |x|^{\alpha + \delta r} dx \Big)^{1/r} \\ &= (1- \varkappa^{-1})\Big( \sum_{k=1}^\infty \varkappa^{(- \frac{\alpha+1}{r} )k}  \Big)\|f \|_{L^r_{\alpha + \delta r} (G)} \\ &= \frac{1- \varkappa^{-1}}{1- \varkappa^{-\frac{\alpha + 1}{r}}} \varkappa^{- \frac{\alpha + 1}{r}}\|f\|_{L^r_{\alpha + \delta r} (G)}.
\end{align*}
The above shows that $$\|H_\delta^* \|_{\mathcal{L}(\Tilde{L}^r_{\alpha + \delta r} (G) ,\Tilde{L}^r_\alpha (G))} \leq \frac{1- \varkappa^{-1}}{1- \varkappa^{ - \frac{\alpha+1}{r}}}\varkappa^{- \frac{\alpha + 1}{r}} = \frac{1 - \varkappa^{-1}}{\varkappa^{\frac{\alpha + 1}{r}} - 1}.$$For the converse inequality take a sequence $\{ g_n\}_{n \in \N_0}$ in $G$ such that $|g_n|=\varkappa^n$ and define $\varepsilon_n := \varkappa^{-n}$. Consider the functions $f_n$ given by \[f_n (x) :=\begin{cases} |x|^{-\frac{\alpha + \delta r + 1}{r} + \varepsilon_n} & \esp \esp \text{if} \esp |x| < 1, \\ 0   & \esp \esp \text{if} \esp |x| \geq 1.\end{cases}\]Clearly $$\| f_n \|_{L^r_{\alpha + \delta r} (G)}^r = \frac{1-\varkappa^{-1}}{\varkappa^{ \varepsilon_n r} -1 },$$ and \[H_\delta f_n (x) =\begin{cases} 0 & \esp \esp \text{if} \esp |x| \geq 1 , \\ |x|^{-\frac{\alpha + 1}{r} + \varepsilon_n} \int_{1< |t| < \frac{1}{|x|}} |t|^{-1 -\frac{\alpha + 1}{r} + \varepsilon_n}  & \esp \esp \text{if} \esp |x| < 1.\end{cases}\]We can estimate \begin{align*}
    \| H_\delta f \|_{L^r_\alpha(G)} &= \Big(\int_{|x| <  1} \Big( |x|^{-\frac{\alpha +1}{r} + \varepsilon_n} \int_{1 < |t| < \frac{1}{|x|}} |t|^{-1-\frac{\alpha+1}{r} + \varepsilon_n} dt\Big)^r |x|^\alpha dx \Big)^{1/r} \\ & \geq \Big(\int_{|x| < |g_n|^{-1}} \Big( |x|^{-\frac{\alpha +1}{r} + \varepsilon_n} \int_{1 < |t| < |g_n|} |t|^{-1-\frac{\alpha+1}{r} - \varepsilon_n} dt\Big)^r |x|^\alpha dx\Big)^{1/r} \\&= \Big( \int_{|x| < |g_n|^{-1}} |x|^{-1 + r \varepsilon_n} dx \Big)^{1/r} \int_{1< |t| < |g_n|} |t|^{-1-\frac{\alpha  +1}{r} + \varepsilon_n} dt \\ &= |g_n|^{-\varepsilon_n} \|f_n\|_{L^r_{\alpha + \delta r} (G)} \int_{1< |t| <|g_n|} |t|^{-1-\frac{\alpha  +1}{r} + \varepsilon_n} dt,
\end{align*}which shows that $$\|H_\delta \|_{\mathcal{L} (L^r_{\alpha + \delta r} (G) , L^r_\alpha (G))} \geq  |g_n|^{-\varepsilon_n} \int_{1< |t| < |g_n|} |t|^{-1-\frac{\alpha +1}{r} + \varepsilon_n} dt.$$Finally, we take the limit as $n \to \infty$ to conclude $$\|H_\delta \|_{\mathcal{L} (L^r_{\alpha + \delta r} (G) , L^r_\alpha (G))} \geq  \int_{ 1< |t| } |t|^{-1-\frac{\alpha +1}{r} } dt = \frac{1- \varkappa^{-1}}{ \varkappa^{\frac{\alpha+1}{r} } - 1}.$$This conclude the proof.
\end{proof}

Now we present the following weak estimates for the adjoint fractional Hardy operator:

\begin{teo}\label{teoweakestadjointHardy}
Let $G$ be a constant-order Vilenkin group, let us say $|G_n/G_{n+1}|=\varkappa$ for all $n \in \Z$. Let $1<r< \infty,$ $1 \leq s < \infty,$ $ - \infty < \delta < 1$ and $\delta r - 1 < \beta < \infty$ be real numbers, and take $\gamma \in \R$ such that $\frac{\beta + 1}{r} - \delta = \frac{\gamma+1}{s}$. Then $$\| H_\delta^* \|_{\mathcal{L}(L^r_\beta (G), L^{s,\infty}_\gamma (G))} = \Big( \frac{1 - \varkappa^{-1}}{1- \varkappa^{-(\gamma + 1)}} \Big)^{1/s} \Big( \frac{1 - \varkappa^{-1}}{\varkappa^{(\frac{\beta + 1}{r } - \delta ) r' } -1}  \Big)^{1/r'},$$where $\frac{1}{r} + \frac{1}{r'} =1.$
\end{teo}

Notice that in the case where $r=s$ and $\delta = 0$ the value of $\gamma$ has to be equal to $\beta$. It follows: 
\begin{coro}
Let $G$ be a constant-order Vilenkin group, let us say $|G_n / G_{n+1}| = \varkappa$ for all $n \in \Z$. Let $1<r< \infty$ and $-1 < \beta < \infty$ be given real numbers. Then: $$\| H^* \|_{\mathcal{L}(L^r_\beta (G), L^{r,\infty}_\beta (G))} = \Big( \frac{1 - \varkappa^{-1}}{1- \varkappa^{-(\beta + 1)}} \Big)^{1/r} \Big( \frac{1 - \varkappa^{-1}}{\varkappa^{\frac{\beta + 1}{r-1} } - 1}  \Big)^{1/r'},$$where $\frac{1}{r} + \frac{1}{r'} =1.$ 
\end{coro}

In the special case  $r=1$ we got a similar estimation: 

\begin{teo}\label{teoweakadjointHardyL1}
Let $G$ be a constant-order Vilenkin group, let us say $|G_n/G_{n+1}|=\varkappa$ for all $n \in \Z$. Let $1 \leq s< \infty,$ $- \infty < \delta < 1$ and $- ( 1 - \delta) < \beta < \infty$ be real numbers, and take $\gamma \in \R$ such that $1 - \delta + \beta = \frac{\gamma+1}{s}$. Then $$\| H_\delta^* \|_{\mathcal{L}(L^1_{\beta} (G), L^{s,\infty}_\gamma (G))} \leq \Big( \frac{1 - \varkappa^{-1}}{1- \varkappa^{-(\gamma + 1)}} \Big)^{1/s}= \Big( \frac{1 - \varkappa^{-1}}{1- \varkappa^{-s(1 - \delta + \beta)}} \Big)^{1/s} .$$
\end{teo}

\begin{proof}[Proof of Theorem \ref{teoweakestadjointHardy}:]
By application of the H{\"o}lder inequality we obtain: \begin{align*}
    |H_\delta^* f(x) | &\leq \int_{|y|>|x|} \frac{|f(y)|}{|y|^{1 - \delta}} dy \\ & \leq \Big( \int_{|y|>|x|} |y|^{-(1 - \delta + \frac{\beta}{r})r'} dy \Big)^{1/r'} \Big( \int_{|y|>|x|} |f(y)|^r |y|^\beta \Big)^{1/r} \\ & \leq \Big( \int_{|y|>|x|} |y|^{-(1 - \delta + \frac{\beta}{r})r'} dy \Big)^{1/r'} \| f \|_{L^r_\beta (G)} \\ &= | x |^{- (\frac{\beta+1}{r} - \delta )} \Big( \frac{1 - \varkappa^{-1}}{\varkappa^{(\frac{\beta + 1}{r} - \delta)r' } -1} \Big)^{1/r'} \| f \|_{L^r_\beta (G)}\\ &= C_f | x |^{- (\frac{\gamma + 1}{s} )},
\end{align*}where $$C_f := \Big( \frac{1 - \varkappa^{-1}}{\varkappa^{(\frac{\beta + 1}{r} - \delta)r'} -1} \Big)^{1/r'} \| f \|_{L^r_\beta (G)}.$$Now as before we define the sets $$A_\lambda := \{x \in G \esp : \esp |H_\delta^* f (x)| > \lambda\} \subset B_\lambda := \Big\{ x \in G \esp : \esp |x|< \Big( \frac{C_f}{\lambda} \Big)^{\frac{s}{\gamma + 1}} \Big\}.$$Hence \begin{align*}
    \| H_\delta^* f \|_{L^{s , \infty}_\gamma (G)} &= \sup_{\lambda>0} \lambda \cdot \Big( \int_{A_\lambda} |x|^\gamma dx  \Big)^{1/s} \\ & \leq  \sup_{\lambda>0} \lambda \cdot \Big( \int_{B_\lambda} |x|^\gamma dx  \Big)^{1/s} \\ &= \Big( \frac{1 - \varkappa^{-1}}{1 - \varkappa^{- (\gamma + 1)}} \Big)^{1/s} \Big( \frac{1 - \varkappa^{-1}}{\varkappa^{(\frac{\beta + 1}{r} - \delta)r'} - 1} \Big)^{1/r'} \| f \|_{L^r_\beta (G)}.  
\end{align*}In order to show the above estimate is sharp just define $$f_0 (x) := |x|^{\frac{\delta -1 - \beta}{r-1}} \mathbb{1}_{G \setminus G_0} (x) , $$and notice that $$\| H_\delta^* f_0 \|_{L^{s , \infty}_\gamma (G)} = \Big( \frac{1 - \varkappa^{-1}}{1 - \varkappa^{- (\gamma + 1)}} \Big)^{1/s} \Big( \frac{1 - \varkappa^{-1}}{\varkappa^{(\frac{\beta + 1}{r} - \delta)r'} - 1} \Big)^{1/r'} \| f_0 \|_{L^r_\beta (G)}.$$This conclude the proof.
\end{proof}

\section{Graded $\K$-Lie Groups}

Now we proceed with the second part of our results. The first one is Theorem \ref{teofirstintegralinequalitycompact} which together with Theorem \ref{Teofirstintegralineqnoncompact} give an analogue of Theorem 3.1 in \cite{2018arXiv180501064R}.

\begin{proof}[Proof of Theorem \ref{teofirstintegralinequalitycompact}]
Let us show first that equation (\ref{eq2}) implies inequality (\ref{eq1}). 

\begin{align*}
    \int_{G} \Big( \int_{B(e, |x|_G)} f(z) dz \Big)^s \varphi (x) dx &= \sum_{m \geq 0 } \int_{G_m \setminus G_{m+1} } \Big( \int_{B(e ,q^{-m})} f(z) dz  \Big)^s \varphi (x) dx \\ &= \sum_{m \geq 0 } \int_{G_m \setminus G_{m+1} } \Big( \sum_{k \geq  m } \int_{G_k \setminus G_{k+1} } f(z) dz  \Big)^s \varphi (x) dx.
\end{align*}
Now we define: 
\begin{itemize}
    \item \begin{align*}
        g(m) &:=\Big( \int_{G_m} \psi (x)^{1-r'} dx \Big)^{1/r r'} = \Big( \sum_{l \geq m} q^{-lQ} \int_{G_0 \setminus G_1} \psi ( q^l x)^{1 -r'} dx  \Big)^{1/r r'}  \\ &= \Big( \sum_{l \geq m} q^{-l(Q + \beta(1-r'))} \int_{G_0 \setminus G_1} \psi (x) dx \Big)^{1/rr'} \\&= \Big( \frac{1}{1 - q^{-(Q + \beta(1-r'))}} \int_{G_0 \setminus G_1} \psi (x)^{1 -r'} dx \Big)^{1/r r'} q^{-m \frac{Q + \beta (1-r')}{rr'}} \\ &= C_1 q^{-m \frac{Q + \beta (1-r')}{rr'}} ; \esp \esp \esp   C_1 :=  \Big( \frac{1}{1 - q^{-(Q + \beta(1-r'))}} \int_{G_0 \setminus G_1} \psi (x)^{1 -r'} dx \Big)^{1/r r'}.
    \end{align*}
    \item \begin{align*}
        U(m) := \int_{G_0 \setminus G_1} q^{-Qm} [f(\mathscr{p}^m x) \psi (\mathscr{p}^m x )^{1/r} g(m)]^r dx = \Big( \int_{G_m \setminus G_{m +1}} f(x)^r \psi (x ) dx\Big) g(m)^r ,
    \end{align*}
    \item \begin{align*}
        V(m) &:= \sum_{l \geq m} q^{-lQ} \int_{G_0 \setminus G_1} [\psi (\mathscr{p}^m x)^{1/r} g(l)]^{-r'} dx \\ &= \sum_{l \geq m } q^{-l(Q- \beta \frac{r'}{r} )} \int_{G_0 \setminus G_1} \psi(x)^{1-r'} \big( C_1 q^{-l \frac{Q + \beta (1-r')}{rr'}} \big)^{-r'} dx  \\&=  \sum_{l \geq m } C_1^{-r'} q^{-l (\frac{Q}{r'} - \frac{\beta}{r})} \int_{G_0 \setminus G_1} \psi(x)^{1 - r'} dx \\ &= \Big( \frac{C_1^{-r'}}{1- q^{- (\frac{Q}{r'} - \frac{\beta}{r})}} \int_{G_0 \setminus G_1} \psi(x)^{1-r'} dx \Big)  q^{-m(\frac{Q}{r'} - \frac{\beta}{r})} \\ &= C_2 q^{-m(\frac{Q}{r'} - \frac{\beta}{r})}; \esp \esp \esp C_2 := \frac{C_1^{-r'}}{1- q^{- (\frac{Q}{r'} - \frac{\beta}{r})}} \int_{G_0 \setminus G_1} \psi(x)^{1-r'} dx. 
    \end{align*}
    \item \begin{align*}
        W(m) &:=\int_{G_0 \setminus G_1} q^{-mQ} \varphi (\mathscr{p}^m x) dx \\& = \int_{G_m \setminus G_{m+1}} \varphi_\alpha (x) dx = q^{m (\alpha - Q)} \int_{G_0 \setminus G_1} \varphi (x) dx  = C_3 q^{m(\alpha - Q)}.
    \end{align*}
\end{itemize} 
By using H{\"o}lder inequality 
\begin{align*}
    \sum_{k \geq m} &\int_{G_0 \setminus G_1} q^{-kQ} f(\mathscr{p}^k z) dz = \sum_{k \geq m} \int_{G_0 \setminus G_1} q^{-kQ/r} [f(\mathscr{p}^k z) \psi (\mathscr{p}^k z)^{1/r} g(k) ] q^{-kQ/r'} [ \psi (\mathscr{p}^k z)^{1/r} g(k)]^{-1} dz \\ & \leq \Big( \sum_{k \geq m} \int_{G_0 \setminus G_1} q^{-kQ} [f(\mathscr{p}^k z) \psi (\mathscr{p}^k z)^{1/r} g(k) ]^r dz \Big)^{1/r} \Big( \sum_{k \geq m} q^{-kQ}
    \int_{G_0 \setminus G_1} [ \psi (\mathscr{p}^k z)^{1/r} g(k)]^{-r'} dz \Big)^{1/r'} \\ &= \Big( \sum_{k \geq m} U(k) \Big)^{1/r} V(m)^{1/r'}. 
\end{align*}
In this way we can write 
$$\int_{G} \Big( \int_{B(e, |x|_G)} f(z) dz \Big)^s \varphi (x) dx \leq \sum_{m \geq 0 } W(m) \Big( \sum_{k \geq m} U(k) \Big)^{s/r} V(m)^{s/r'}.$$
Hence, applying the Minkowski inequality in the right side of the above inequality, we get: 
 $$\int_{G} \Big( \int_{B(e, |x|_G)} f(z) dz \Big)^s \varphi (x) dx \leq \Big( \sum_{k \geq 0} U(k) \Big( \sum_{0 \leq m \leq k} W(m) V(m)^{s/r'} \Big)^{r/s} \Big)^{s/r}.$$
 
Now we estimate the right hand side of the above expression: 
\begin{align*}
    \sum_{0 \leq m \leq k} W(m) V(m)^{s/r'} & =  \sum_{0 \leq m \leq k } C_3 q^{m(\alpha - Q)} \Big(C_2 q^{-m\big( \frac{Q}{r'} - \frac{\beta}{r} \big)} \Big)^{s/r'} \\ &= C_3 C_2^{s/r'} q^{k  \frac{s}{r'}\big( \frac{r'(\alpha -Q)}{s} - \big( \frac{Q}{r'} - \frac{\beta}{r} \big) \big)} \sum_{-k \leq m \leq 0} q^{m \frac{s}{r'}\big( \frac{r'(\alpha -Q)}{s} - \big( \frac{Q}{r'} - \frac{\beta}{r} \big) \big)} \\ & \leq \frac{C_3 C_2^{s/r'}}{1 - q^{- \frac{s}{r'}\big( \frac{r'(\alpha -Q)}{s} - \big( \frac{Q}{r'} - \frac{\beta}{r} \big) \big)}} q^{k \frac{s}{r'}\big( \frac{r'(\alpha -Q)}{s} - \big( \frac{Q}{r'} + \frac{\beta}{r} \big) \big)} \\ &= C_4 q^{k \frac{s}{r'}\big( \frac{r'(\alpha -Q)}{s} - \big( \frac{Q}{r'} - \frac{\beta}{r} \big) \big)}; \esp \esp \esp C_4:=\frac{C_3 C_2^{s/r'}}{1 - q^{- \frac{s}{r'}\big( \frac{r'(\alpha -Q)}{s} - \big( \frac{Q}{r'} - \frac{\beta}{r} \big) \big)}}.
\end{align*}Finally we have 
\begin{align*}
    \int_{G} \Big( \int_{B(e, |x|_G)} &f(z) dz \Big)^s \varphi (x) dx \leq \Big( \sum_{k \geq 0} U(k) \Big( \sum_{0 \leq m \leq k} W(m) V(m)^{s/r'} \Big)^{r/s} \Big)^{s/r} \\ &\leq \Big( \sum_{k \geq 0} \int_{G_k \setminus G_{k +1}} f(x)^r \psi(x ) dx \cdot C_1 q^{-k \big( \frac{Q}{r'} - \frac{\beta}{r} \big)}  \big( C_4 q^{k \frac{s}{r'}\big( \frac{r'(\alpha -Q)}{s} - \big( \frac{Q}{r'} - \frac{\beta}{r} \big) \big)} \big)^{r/s}  \Big)^{s/r}\\ &=\Big(C_1 C_4^{r/s} \sum_{k \geq 0} \int_{G_k \setminus G_{k +1}} f(x)^r \psi(x ) dx \cdot  q^{k \big(  \frac{Q}{r'} - \frac{\beta}{r}  + \frac{r}{r'}\big( \frac{r'(\alpha -Q)}{s} - \big( \frac{Q}{r'} - \frac{\beta}{r} \big)\big)}  \Big)^{s/r} \\ & = C_1^{s/r} C_4 \Big(\sum_{k \geq 0} \int_{G_k \setminus G_{k+1} } f(x)^r \psi (x) dx \Big)^{s/r} \\ &= C_1^{s/r} C_4 \Big( \int_G f(x)^r \psi (x) dx \Big)^{s/r},
\end{align*}where in the above we used the fact that $$  \frac{Q}{r'} - \frac{\beta}{r}  + \frac{r}{r'}\big( \frac{r'(\alpha -Q)}{s} - \big( \frac{Q}{r'} - \frac{\beta}{r} \big) \big) = r \Big( \frac{\alpha -Q}{s} - \big( \frac{Q}{r'} - \frac{\beta}{r}\big) \Big)=0.$$This shows how equation (\ref{eq2}) is a sufficient for inequality (\ref{eq1}) to hold. Now we show that inequality (\ref{eq1}) implies equation \ref{equat3}. For that purpose we define the sequence of functions $\{f_n \}_{n \in \N}$ by $$f_n(x) := \psi(x)^{1-r'} \mathbb{1}_{G_n} (x).$$We can check that $$\Big( \int_G f_n(x)^r \psi(x) dx \Big)^{1/r} = \Big( \int_{G_n} \psi(x)^{1-r'} dx \Big)^{1/r},$$and consequently:
\begin{align*}
    C= &C \Big( \int_G f_n(x)^r \psi(x) dx \Big)^{1/r} \Big( \int_{G_n} \psi(x)^{1-r'} dx \Big)^{-1/r} \\ &\geq \Big( \int_G \phi(x) \Big( \int_{B(e,|x|)} f_n(z) dz \Big)^s dx \Big)^{1/s}\Big( \int_{G_n} \psi(x)^{1-r'} dx \Big)^{-1/r} \\ &\geq \Big( \int_{G \setminus G_n} \phi(x) \Big( \int_{B(e,|x|)} f_n(z) dz \Big)^s dx \Big)^{1/s}\Big( \int_{G_n} \psi(x)^{1-r'} dx \Big)^{-1/r} \\ &=\Big( \int_{G \setminus G_n} \phi(x)  dx \Big)^{1/s}\Big( \int_{G_n} \psi(x)^{1-r'} dx \Big)^{1/r'}.
\end{align*}
Computing we get: $$\Big( \int_{G_n} \psi (x)^{1 -r' } dx \Big)^{1/r'} = \Big( \frac{1}{1- q^{-(Q + \beta(1-r'))}} \int_{G_0 \setminus G_1} \psi (x)^{1 - r'} dx\Big)^{1/r'} q^{-n \big( \frac{Q}{r'} - \frac{\beta}{r} \big)},$$ and $$\Big( \int_{G \setminus G_n} \varphi(x)  dx \Big)^{1/s} = \Big( q^{Q - \alpha} \frac{1 - q^{-n (\alpha - Q)}}{1 - q^{-(\alpha - Q)}} \int_{G_0 \setminus G_1} \varphi (x) dx \Big)^{1/s} q^{n \big( \frac{\alpha - Q}{s} \big)}.$$Now it is obvious that the inequality  $$\Big( \int_{G \setminus G_n} \phi(x)  dx \Big)^{1/s}\Big( \int_{G_n} \psi(x)^{1-r'} dx \Big)^{1/r'} \leq C,$$holds for every $n \in \N_0$ only if $$\frac{\alpha - Q}{s} - \big( \frac{Q}{r'} + \frac{\beta}{r} \big) \leq 0.$$This concludes the proof. 
\end{proof}

In the non-compact case we can use the same arguments to provide weighted $L^r-L^s$ integral inequalities, but this time we are able to provide a necessary and sufficient conditions.

\begin{rem}
In the case where $\varphi(x) = |x|^{-Q s} = | B(e, |x|)|^{-s} $ and $\psi (x) =1$ we obtain the $L^r-L^s$ boundedness of the Hardy operator on $G$.
\end{rem}
\begin{proof}[Proof of Theorem \ref{Teofirstintegralineqnoncompact}:]
Let us show first that Equation \ref{eq4} implies Equation \ref{eq3} in the same way we did it for the proof of Theorem \ref{Teofirstintegralineqnoncompact}. 

\begin{align*}
    \int_{G} \Big( \int_{B(e, |x|)} f(z) dz \Big)^s \varphi (x) dx &= \sum_{m \in \Z } \int_{G_m \setminus G_{m+1} } \Big( \int_{B(e ,q^{-m})} f(z) dz  \Big)^s \varphi (x) dx \\ &= \sum_{m \in \Z } \int_{G_m \setminus G_{m+1} } \Big( \sum_{k \geq  m } \int_{G_k \setminus G_{k+1} } f(z) dz  \Big)^s \varphi (x) dx.
\end{align*}
Now we define $g, U, V$ and $W$ in the same way as before, so that: 
\begin{itemize}
    \item $g(m) = C_1 q^{-m \frac{Q + \beta (1-r')}{rr'}}$,
    \item $U(m) = \Big( \int_{G_m \setminus G_{m +1}} f(x)^r \psi_\beta (x ) dx\Big) g(m)^r,$
    \item $V(m) = C_2 q^{-m(\frac{Q}{r'} - \frac{\beta}{r})},$
    \item $W(m) = C_3 q^{m(\alpha - Q)}.$
\end{itemize} 
By using H{\"o}lder inequality 
\begin{align*}
    \sum_{k \geq m} &\int_{G_0 \setminus G_1} q^{-kQ} f(\mathscr{p}^k z) dz = \sum_{k \geq m} \int_{G_0 \setminus G_1} q^{-kQ/r} [f(\mathscr{p}^k z) \psi (\mathscr{p}^k z)^{1/r} g(k) ] q^{-kQ/r'} [ \psi (\mathscr{p}^k z)^{1/r} g(k)]^{-1} dz \\ & \leq \Big( \sum_{k \geq m} \int_{G_0 \setminus G_1} q^{-kQ} [f(\mathscr{p}^k z) \psi (\mathscr{p}^k z)^{1/r} g(k) ]^r dz \Big)^{1/r} \Big( \sum_{k \geq m} q^{-kQ}
    \int_{G_0 \setminus G_1} [ \psi (\mathscr{p}^k z)^{1/r} g(k)]^{-r'} dz \Big)^{1/r'} \\ &= \Big( \sum_{k \geq m} U(k) \Big)^{1/r} V(m)^{1/r'}. 
\end{align*}
In this way we can write 
$$\int_{G} \Big( \int_{B(e, |x|_G)} f(z) dz \Big)^s \varphi (x) dx \leq \sum_{m \in \Z } W(m) \Big( \sum_{k \geq m} U(k) \Big)^{s/r} V(m)^{s/r'}.$$
As before, applying the Minkowski inequality in the right side of the above inequality, we get: 
 $$\int_{G} \Big( \int_{B(e, |x|_G)} f(z) dz \Big)^s \varphi (x) dx \leq \Big( \sum_{k \in \Z} U(k) \Big( \sum_{m \leq k} W(m) V(m)^{s/r'} \Big)^{r/s} \Big)^{s/r}.$$
 
Now we estimate the right hand side of the above expression: 
\begin{align*}
    \sum_{ m \leq k} W(m) V(m)^{s/r'} & =  \sum_{m \leq k } C_3 q^{k(\alpha - Q)} \Big(C_2 q^{-m\big( \frac{Q}{r'} - \frac{\beta}{r} \big)} \Big)^{s/r'} \\ &= C_3 C_2^{s/r'} q^{k  \frac{s}{r'}\big( \frac{r'(\alpha -Q)}{s} - \big( \frac{Q}{r'} - \frac{\beta}{r} \big) \big)} \sum_{ m \leq 0} q^{m \frac{s}{r'}\big( \frac{r'(\alpha -Q)}{s} - \big( \frac{Q}{r'} - \frac{\beta}{r} \big) \big)} \\ & = \frac{C_3 C_2^{s/r'}}{1 - q^{- \frac{s}{r'}\big( \frac{r'(\alpha -Q)}{s} - \big( \frac{Q}{r'} - \frac{\beta}{r} \big) \big)}} q^{k \frac{s}{r'}\big( \frac{r'(\alpha -Q)}{s} - \big( \frac{Q}{r'} - \frac{\beta}{r} \big) \big)} \\ &= C_4 q^{k \frac{s}{r'}\big( \frac{r'(\alpha -Q)}{s} - \big( \frac{Q}{r'} - \frac{\beta}{r} \big) \big)}; \esp \esp \esp C_4:=\frac{C_3 C_2^{s/r'}}{1 - q^{- \frac{s}{r'}\big( \frac{r'(\alpha -Q)}{s} - \big( \frac{Q}{r'} - \frac{\beta}{r} \big) \big)}}.
\end{align*}Finally we have 
\begin{align*}
    \int_{G} \Big( \int_{B(e, |x|_G)} &f(z) dz \Big)^s \varphi (x) dx \leq \Big( \sum_{k \in \Z} U(k) \Big( \sum_{ m \leq k} W(m) V(m)^{s/r'} \Big)^{r/s} \Big)^{s/r} \\ &\leq \Big( \sum_{k \in \Z} \int_{G_k \setminus G_{k +1}} f(x)^r \psi_\beta (x ) dx \cdot C_1 q^{-k \big( \frac{Q}{r'} + \frac{\beta}{r} \big)}  \big( C_4 q^{k \frac{s}{r'}\big( \frac{r'(\alpha -Q)}{s} - \big( \frac{Q}{r'} + \frac{\beta}{r} \big) \big)} \big)^{r/s}  \Big)^{s/r}\\ & \leq C_1^{s/r} C_4 \Big(\sum_{k \in \Z} \int_{G_k \setminus G_{k+1} } f(x)^r \psi (x) dx \Big)^{s/r} \\ &= C_1^{s/r} C_4 \Big( \int_G f(x)^r \psi (x) dx \Big)^{s/r}.
\end{align*}
This shows that Equation \ref{eq4} implies Equation \ref{eq3}. To prove the converse we just take again $f_n (x) =  \psi (x)^{1 - r'} \mathbb{1}_{G_n} (x)$ so that $$\Big( \int_G f_n(x)^r \psi(x) dx \Big)^{1/r} \Big( \int_{G_n} \psi(x)^{1-r'} dx \Big)^{-1/r} =1.$$In this way:
\begin{align*}
    C= &C \Big( \int_G f_n(x)^r \psi(x) dx \Big)^{1/r} \Big( \int_{G_n} \psi(x)^{1-r'} dx \Big)^{-1/r} \\ &\geq \Big( \int_G \phi(x) \Big( \int_{B(e,|x|)} f_n(z) dz \Big)^s dx \Big)^{1/s}\Big( \int_{G_n} \psi(x)^{1-r'} dx \Big)^{-1/r} \\ &\geq \Big( \int_{G \setminus G_n} \phi(x) \Big( \int_{B(e,|x|)} f_n(z) dz \Big)^s dx \Big)^{1/s}\Big( \int_{G_n} \psi(x)^{1-r'} dx \Big)^{-1/r} \\ &=\Big( \int_{G \setminus G_n} \phi(x)  dx \Big)^{1/s}\Big( \int_{G_n} \psi(x)^{1-r'} dx \Big)^{1/r'}.
\end{align*}
Computing we get: $$\Big( \int_{G_n} \psi (x)^{1 -r' } dx \Big)^{1/r} = \Big( \frac{1}{1- p^{-(Q + \beta(1-r'))}} \int_{G_0 \setminus G_1} \psi (x)^{1 - r'} dx\Big)^{1/r} q^{-n \big( \frac{Q}{r'} - \frac{\beta}{r} \big)},$$ and $$\Big( \int_{G \setminus G_n} \phi(x)  dx \Big)^{1/s} = \Big( q^{Q - \alpha} \frac{1 - q^{-n (\alpha - Q)}}{1 - q^{-(\alpha - Q)}} \int_{G_0 \setminus G_1} \Big)^{1/s} q^{n \big( \frac{\alpha - Q}{s} \big)}.$$Now it is obvious that the inequality  $$\Big( \int_{G \setminus G_n} \phi(x)  dx \Big)^{1/s}\Big( \int_{G_n} \psi(x)^{1-r'} dx \Big)^{1/r'} \leq C,$$holds for every $n \in \Z$ only if $$\frac{\alpha - Q}{s} - \big( \frac{Q}{r'} - \frac{\beta}{r} \big) = 0.$$This concludes the proof. 
\end{proof}
With similar arguments we can provide weighted $L^r - L^s$ integral inequalities for the adjoint Hardy operator stated in Theorem \ref{Teosecondintegralineqnoncompact}.
\begin{rem}
It is known that for real Lie groups that the following condition is necessary and sufficient for the $L^r-L^s$ weighted integral inequality (\ref{eq3}) to hold: 
$$\sup_{R >0}\Big( \int_{|x| > R} \varphi(x) dx \Big)^{1/s} \Big( \int_{|x| \leq R} \psi(x)^{1-r'} dx \Big)^{-1/r} < \infty.$$We would like to point to the fact that the same is true here, but since we are considering exclusively homogeneous functions this condition transforms into condition (\ref{eq4}). 
\end{rem}

\begin{proof}[Proof Theorem \ref{Teofunctionalineq}:]
The first step is to see that $|K_a (x)| \leq C_0 |x|_G^{a-Q}$ implies the following inequality: 
$$\Big\| \frac{f * K_a}{|x|_G^{b/s}}  \Big\|_{L^s (G)} \leq C_0 \Big\| \frac{f * | \cdot |_G^{a - Q}}{|x|^{b/s}_G}  \Big\|_{L^s (G)},$$where \begin{align*}
    (f* |\cdot|^{a-Q}_G)(x) &= \int_G f(y) |y^{-1} x|^{a-Q}_G dy \\ &= \int_{|x|_G>|y|_G} f(y) |y^{-1} x|^{a-Q}_G dy+ \int_{|y|_G>|x|_G} f(y) |y^{-1} x|^{a-Q}_G dy + \int_{|x|_G = |y|_G} f(y) |y^{-1} x|^{a-Q}_G dy.
\end{align*}
Now we proceed by parts:
\begin{itemize}
    \item If $|x|_G>|y|_G$ then $|y^{-1}x|_G = |x|_G$. In this way \begin{align*}
        M_1^s &:=\int_{G} |x|^{-b}_G \Big( \int_{|x|_G>|y|_G} f(y) |y^{-1} x|_G^{a-Q} dy \Big)^s dx \\ &= \int_{G} |x|^{(a-Q)s-b}_G \Big( \int_{|x|_G>|y|_G} f(y) dy \Big)^s dx.
    \end{align*}
    Applying Theorem \ref{Teofirstintegralineqnoncompact} with $\alpha = (Q - a)s +b$ and $\beta=0$ we can conclude that $$M_1 \leq C \| f \|_{L^r (G)}.$$
    \item If $|y|_G>|x|_G$ then $|y^{-1} x|_G = |y|_G$. Hence \begin{align*}
        M_2^s &:= \int_{G} |x|^{-b}_G \Big( \int_{|y|_G>|x|_G} f(y) |y^{-1} x|^{a-Q}_G dy \Big)^s dx \\ &= \int_{G} |x|^{-b}_G \Big( \int_{|y|_G>|x|_G} f(y) |y|_G^{a-Q} dy \Big)^s dx \\ &= \int_{G} |x|^{-b}_G \Big( \int_{|y|_G>|x|_G} \Tilde{f}(y) dy \Big)^s dx,
    \end{align*}where $\Tilde{f}(y) := f(y) |y|^{a-Q}_G$. Now put $\varphi(x) = |x|^{-b}_G$ and $\psi(x) = |x|_G^{(Q -a)r}$. Then by using Theorem \ref{Teosecondintegralineqnoncompact} we get \begin{align*}
        M_2 \leq C \Big( \int_G \Tilde{f}(x)^r \psi(x) dx \Big)^{1/r} = C \| f\|_{L^r (G)}. 
    \end{align*}
    \item If $|x|_G=|y|_G$ we have \begin{align*}
        M_3^s&:= \int_G |x|^{-b}_G \Big( \int_{|y|_G=|x|_G} f(y) |y^{-1}x|^{a-Q}_G dy\Big)^s dx \\ & = \sum_{k \in \Z} q^{kb} \int_{|x|_G=q^{-k}} \Big( \int_{|y|=q^{-k}} |y^{-1}x|^{a-Q}_G f(y)dy \Big)^s dx.
    \end{align*}Let $\mathbb{1}_A$ be the characteristic function of $A \subseteq G$. Let $t$ be a positive real number such that $1 + \frac{1}{t} = \frac{1}{s} + \frac{1}{r}$. Then, $$\int_{|y|_G=q^{-k}} |y^{-1}x|_G^{a-Q} f(y)dy = \int_{G} [|y^{-1}x|^{a-Q} \mathbb{1}_{G_k} (y^{-1} x)] [ f(y) \mathbb{1}_{G_k \setminus G_{k+1}} (y)]dy,$$so that by Young's convolution inequality we get $$M_3^s \leq \sum_{k \in \Z} q^{kb} \|( |\cdot|^{a-Q}_G \mathbb{1}_{G_k}) \|_{L^t(G)}^t \| (f \esp  \mathbb{1}_{G_k \setminus G_{k+1}} ) \|_{L^r(G)}^s, \esp \esp \esp t= \frac{sr}{sr +r -s }.$$Now we compute: \begin{align*}
        \|( |\cdot|^{a-Q} \mathbb{1}_{G_k}) \|_{L^t(G)}^t  &= \int_{G_k} |x|^{(a -Q)t} dx \\ &=(1-q^{-Q}) \sum_{l \in \N_0} q^{-(k+l)((a-Q)t +Q) } \\ &=(1-q^{-Q}) \sum_{l \in \N_0} q^{-(k+l)\frac{br}{sr +r -s }  } \\ &=\frac{1-p^{-Q}}{1 - q^{-\frac{br}{sr +r -s } }} p^{-k\frac{br}{sr +r -s }},
    \end{align*}which implies that $$\|( |\cdot|^{a-Q} \mathbb{1}_{G_k}) \|_{L^t(G)}^t = \Big( \frac{1-q^{-Q}}{1 - q^{-\frac{br}{sr +r -s } }} \Big)^{s \frac{sr+r-s}{sr}} p^{-kb}.$$Now it is clear that \begin{align*}
        M_3^s &\leq C \sum_{k \in \Z} \Big( \int_{G_k \setminus G_{k+1} } f(x)^r dx \Big)^{s/r} \leq C \| f \|_{L^r (G)}.
    \end{align*}
\end{itemize}
This completes the proof.
\end{proof}
\begin{proof}[Proof of Theorem \ref{UP}]
By using H\"{o}lder's inequality and Hardy's inequality, we get 
\begin{equation}
    \begin{split}
        \int_{G}|f(x)|^{2}dx&=\int_{G}|x|_G^{-a}|f(x)| |f(x)||x|^{a}_G dx\\&\leq \left(\int_{G}|x|^{-ar}_G|f(x)|^{r}dx\right)^{\frac{1}{r}} \left(\int_{G}|x|^{ar'}_G|f(x)|^{r'}dx\right)^{\frac{1}{r'}}\\&
        \leq C \| \MO^a f \|_{L^r (G)} \||\cdot|^{a}_G f\|_{L^{r'}(G)}.
    \end{split}
\end{equation}
\end{proof}

Let us prove the Gagliardo-Nirenberg inequality.
\begin{proof}[Proof of Theorem \ref{GNthm}]
By using H\"{o}lder's inequality  $\frac{1}{s}=\alpha\left(\frac{1}{r}-\frac{a}{Q}\right)+\frac{1-\alpha}{\tau}$ and the  Sobolev inequality, we get 
\begin{equation}
    \begin{split}
        \int_{G}|f(x)|^{s}dx&=\int_{G}|f(x)|^{(1-\alpha)s}|f(x)|^{\alpha s}dx\\&
        \leq \left(\int_{G}|f(x)|^{\frac{Qr}{Q-a r}}dx\right)^{\alpha s\left(\frac{1}{r}-\frac{a}{Q}\right)}\left(\int_{G}|f(x)|^{\tau}dx\right)^{\frac{(1-\alpha)s}{\tau}}\\&
           \leq C\| \MO^a f \|^{\alpha}_{L^r (G)}\|f\|^{1-\alpha}_{L^{\tau}(G)}.
    \end{split}
\end{equation}
\end{proof}
Now we conclude this section by proving the Stein-Weiss inequality on graded $p$-adic Lie groups. We are interested now in studying the following integral operator: $$I_\lambda f(x) := \int_{G} \frac{f(y)}{|y^{-1}x|^{\lambda}_G} dy, \esp \esp 0< \lambda < Q.$$ 

\begin{proof}[Proof of Theorem \ref{teosteinweiss}:]
We will use the same arguments as in the proof of Theorem \ref{Teofunctionalineq}. Let us start by dividing our integral in three parts: 
\begin{align*}
    \| |\cdot|^{-\alpha}_G I_\lambda f \|_{L^s(G)}^s &= \int_G\Big( \int_{G} \frac{f(y)}{|x|^\alpha_G |y^{-1}x|^{\lambda}_G} dy\Big)^s dx \leq 3^s( I_1 + I_2 + I_3),
\end{align*}
where $$I_1=\int_G\Big( \int_{|y|_G < |x|_G} \frac{f(y)}{|x|^\alpha_G |y^{-1}x|^{\lambda}_G} dy\Big)^s dx= \int_G |x|^{-s(\alpha+\lambda)}_G \Big( \int_{|y|_G<|x|_G} f(y)  dy\Big)^s dx,$$ $$I_2=\int_G\Big( \int_{|x|_G < |y|_G} \frac{f(y)}{|x|^\alpha_G |y^{-1}x|^{\lambda}_G} dy\Big)^s dx= \int_G |x|^{-\alpha s}_G\Big( \int_{|x|_G < |y|_G} \frac{f(y)}{ |y|^{\lambda}_G} dy\Big)^s dx, $$and $$I_3= \int_G\Big( \int_{|x|_G = |y|_G} \frac{f(y)}{|x|^\alpha_G |y^{-1}x|^{\lambda}_G} dy\Big)^s dx.$$Now we proceed to estimate the above integrals: 
\begin{itemize}
    \item For $I_1$ we only need to apply Theorem \ref{Teofirstintegralineqnoncompact}. In order to do so just check that $$\frac{s(\alpha + \lambda) - Q}{s}=Q(\frac{\alpha + \lambda}{Q} - \frac{1}{s}) =Q( 1 - \frac{1}{r} - \frac{\beta}{Q}) = \frac{Q}{r'} - \beta>0.$$
    \item For $I_2$ we apply Theorem \ref{Teosecondintegralineqnoncompact}  for $\varphi(x) = |x|^{-\alpha s}_G$ and $\psi(x)= |x|^{(\lambda + \beta)r}_G$. Here we use the fact that $\alpha< Q/s$ implies $Q-r'(\beta + \lambda) <0$. 
    \item For $I_3$ let $s$ be a positive real number such that $1 + \frac{1}{s} = \frac{1}{s} + \frac{1}{r}$. Then, $$\int_{|y|_G=p^{-k}} |y^{-1}x|^{-\lambda}_G f(y)dy = \int_{G} [|y^{-1}x|^{-\lambda}_G \mathbb{1}_{G_k} (y^{-1} x)] [ f(y) \mathbb{1}_{G_k \setminus G_{k+1}} (y)]dy,$$so that by Young's convolution inequality we get $$I_3 \leq \sum_{k \in \Z} p^{k\alpha s} \|( |\cdot|^{ -\lambda}_G \mathbb{1}_{G_k}) \|_{L^s(G)}^s \| (f \esp  \mathbb{1}_{G_k \setminus G_{k+1}} ) \|_{L^r(G)}^s, \esp \esp \esp s= \frac{sr}{sr +r -s }.$$Now we compute: \begin{align*}
        \|( |\cdot|^{-\lambda}_G \mathbb{1}_{G_k}) \|_{L^s(G)}^s  &= \int_{G_k} |x|^{-\lambda s}_G dx \\ &=(1-p^{-Q}) \sum_{l \in \N_0} p^{-(k+l)( -\lambda s +Q) } \\  &=\frac{1-p^{-Q}}{1 - p^{-(-\lambda s +Q) }} p^{-k( -\lambda s +Q)} \\ &=, \frac{1-p^{-Q}}{1 - p^{-s(\beta + \alpha) }} p^{-ks( \alpha + \beta)}
    \end{align*}which implies that $$\|( |\cdot|^{a-Q}_G \mathbb{1}_{G_k}) \|_{L^s(G)}^s = \Big( \frac{1-p^{-Q}}{1 - p^{-s(\alpha + \beta ) }} \Big)^{s/s}  p^{-ks( \alpha +\beta)}.$$Now it is clear that \begin{align*}
        I_3 &\leq C \sum_{k \in \Z} \Big( \int_{G_k \setminus G_{k+1} } f(x)^r |x|^{\beta r}_G dx \Big)^{s/r} \leq C \| |\cdot |^\beta_G f \|_{L^r (G)}.
    \end{align*}
\end{itemize}
This completes the proof.

\end{proof}

\nocite{*}
\bibliographystyle{acm}
\bibliography{main}
\Addresses

\end{document}